\documentclass[11pt]{amsart}

\usepackage{amsthm,amsfonts,amssymb,amsmath}
\usepackage{caption,graphicx}
\usepackage{paralist}

\usepackage{fullpage}
\usepackage{hyperref}
\usepackage{amsthm,amsfonts,amssymb,amsmath,oldgerm}
\usepackage[T1]{fontenc}
\usepackage{lmodern} 
\usepackage[francais,english]{babel}
\usepackage{amsmath,amsthm}
\usepackage{mathrsfs}
\usepackage{bbm} 
\usepackage{amssymb}
\usepackage{amsfonts}
\usepackage{pifont}
\usepackage{stmaryrd}
\usepackage{dsfont}
\usepackage{graphicx,pstricks}
\usepackage{boites,boites_exemples}
\usepackage{enumerate}
\usepackage{lscape}
\usepackage[all]{xy}
\input xy
\xyoption{all}

\usepackage{hyperref}

\newcommand{\be}{\begin{enumerate}}
\newcommand{\ee}{\end{enumerate}}
\newcommand{\pa}{\partial}

\numberwithin{equation}{section}

\newcommand{\R}{\mathbb{R}}

\newcommand{\T}{\mathbb{T}}

\newcommand{\na}{\nabla}

\newcommand{\X}{Y}
\newcommand{\V}{W}

\newtheorem{thm}{Theorem}[section]

\newtheorem{coro}{Corollary}[section]
\newtheorem{prop}{Proposition}[section]
\newtheorem{lem}{Lemma}[section]
\newtheorem{rem}{Remark}[section]

\newcommand{\eps}{\varepsilon}

\begin{document}

\title{Asymptotic stability of equilibria for screened Vlasov-Poisson systems via pointwise dispersive estimates}

\author{Daniel Han-Kwan \and Toan T. Nguyen\and Fr\'ed\'eric Rousset}

\renewcommand{\thefootnote}{\fnsymbol{footnote}}



\maketitle

\begin{abstract}

We revisit the proof of  Landau damping  near stable homogenous equilibria of Vlasov-Poisson systems with screened interactions in the whole space $\R^d$ (for $d\geq3$)
that was first established by Bedrossian, Masmoudi and Mouhot in \cite{BMM2}. Our proof  follows a Lagrangian approach and relies on precise pointwise in time  dispersive estimates in the physical space for the linearized problem that should be of independent interest. This allows to cut down the smoothness of the initial data required in \cite{BMM2} (roughly, we only need Lipschitz regularity).
 Moreover,  the time decay estimates we prove  are essentially sharp, being the same as those for  free transport, up to a logarithmic
 correction. 
\end{abstract}


\section{Introduction}
In this paper, we are interested in the large time behavior of solutions to the Vlasov-Poisson system  with screening 
\begin{equation}\label{iVP}
\left \{ 
\begin{aligned}
&\pa_t f_{i} + v\cdot \na_x f_{i} + E \cdot \na_v f_{i} = 0,\\
&E = -\nabla_x (1-  \Delta_x)^{-1} ( \rho_{i}- 1), \qquad \rho_{i}(t,x)= \int_{\R^d} f_{i} (t,x,v)\, dv,
\end{aligned}
\right.
\end{equation}
on the whole space $x\in \R^d, v\in \R^d$, $d\ge 3$, where $f_i = f_i(t,x,v)\ge 0$ and $E = E(t,x)$. The screening effect comes from the fact that the interaction potential associated to $\text{Id}-\Delta$ is exponentially decaying as opposed to the Coulomb potential associated to $-\Delta$. 
 This system is sometimes referred to as Vlasov-Yukawa and can also be seen as the Vlasov-Poisson system describing the dynamics of ions, in a background of electrons that satisfy a linearization of the Maxwell-Boltzmann law (we refer for example to  \cite{Bouchut,DHK1,ClaudeMB}).  

The global regularity of finite energy solutions for the  Vlasov-Poisson system in the case of three or lower spatial dimension is by now classical (\cite{Pfaffelmoser,Schaeffer91, Horst93, Glassey-book,Lions-Perthame}). The asymptotic behavior of solutions
 for initial data  near  the trivial  equilibrium $0$ has also been the topic of many studies.
  This was first   established in dimension $d \geq 3$ in the unscreened case by Bardos and Degond in \cite{Bardos-Degond}, following a Lagrangian approach. More recently the sharp
   faster decay of derivatives was established in \cite{Hwang}.  This was extended to \eqref{iVP} in dimension $d \geq 2$
    in \cite{Choi} making use of the better decay of the electric field in the case of screened interactions. Other approaches based on vector fields   \cite{Smu} or Fourier analysis \cite{Wang} (space-time resonances)  were also developed recently.

We are interested in the stability and the large time behavior of solutions near spatially homogeneous stationary states $\mu(v)$ such that $\int_{\R^d} \mu(v) \, dv =1$. Namely, we look for a solution under the form
 $ f_{i}(t,x,v)= \mu(v)+ f(t,x,v)$, where $f$ solves the perturbed system 
\begin{equation}
\label{VP}
\left \{ 
\begin{aligned}
&\pa_t f + v\cdot \na_x f +  E \cdot \na_v \mu = -  E \cdot \na_v f , \\
&E = -\nabla_x (1-  \Delta_x)^{-1} \rho, \qquad \rho(t,x)= \int_{\R^d} f (t,x,v)\, dv,\\
&f|_{t=0} = f_0.
\end{aligned}
\right.
\end{equation}
For a class of stable equilibria, we shall study the large time behavior of solutions to \eqref{VP} for suitably small initial data $f_0$. The dynamics of solutions is  expected to asymptotically approach that  of solutions to free transport, a scattering phenomenon that is often referred to as {\em Landau damping}.   Landau damping was proved on the torus $\T^d\times \R^d$ ($d\geq 1$) for data with Gevrey regularity \cite{MV,BMM1}, while on the whole space $\R^3\times \R^3$ it was recently established for the screened Vlasov-Poisson system \eqref{VP} by Bedrossian, Masmoudi and Mouhot in \cite{BMM2}, for data with finite Sobolev regularity. The proof in \cite{BMM2} is inspired by that for the torus case. Though dispersion on the whole space is used  at some crucial points in order to close the estimates in finite regularity, the approach is much more related to the one of \cite{BMM1} for the torus than that of Bardos-Degond (it is actually dispersive properties of the free transport  in the  frequency space  that are used to control the so-called echoes of \cite{MV}).

In this paper, we prove Landau damping and derive dispersive estimates for solutions to \eqref{VP} via a Lagrangian approach that is closer to the Bardos-Degond analysis \cite{Bardos-Degond} for the $\mu=0$ case.  Roughly speaking,
 after proving precise pointwise estimates for the linearized equation, the proof of nonlinear stability can be obtained almost
 in the same way as in the $\mu=0$ case. 
This will allow to strongly cut down the needed regularity on the initial data, as compared to \cite{BMM2}.

Let us now specify our assumptions on the equilibrium $\mu$ and state the main result of this paper. We assume 

 \begin{itemize}
 \item {\bf (H1)} $\mu\in L^1(\mathbb{R}^d)$ is a smooth decaying function satisfying $ \langle v \rangle^k  \nabla_{v} \mu  \in W^{2, \infty}$ and
 $\langle v \rangle^{4d+6}   \na_v \mu  \in W^{2d+5,1}$
for some $k>d$. 
 \item{\bf (H2)} $\mu$ satisfies the Penrose stability criterion: 
$$
\inf_{\gamma\geq 0} \inf_{\tau \in \R, \, \xi \in \R^d} \left| 1 - \int_0^{+\infty} e^{-(\gamma +i \tau)s}  { 1 \over 1 + |\xi|^2} i \xi \cdot \widehat{\na_v \mu} (s \xi) \, ds\right| \geq \kappa, 
$$
for some constant $\kappa>0$, where $\widehat{\na_v \mu}$ is the Fourier transform of $\na_v \mu$ (see~\eqref{fourier-space} for the convention we use). 
\end{itemize}

%

\subsection{Main result}

Our main result is as follows. We recall that we consider $d\geq 3$.

\begin{thm}
\label{theomain}  Assume that (H1) and (H2) are satisfied. 

  Let $k>d$ and $\sigma,p  \in (1,+\infty)$ satisfying $p(\sigma-1)>2d$.  Let $f_0\in W^{1,\infty} \cap W^{1,1}$ be an initial condition for~\eqref{VP} with 
  \begin{equation}
   \label{extra}
   \| \langle v\rangle^{k/p'} f_{0}\|_{W^{\sigma, p}} <+\infty, \qquad \text{for some  } \sigma,p \in (1,+\infty) \text{   satisfying   } p(\sigma-1)>2d
\end{equation}
   and
   \begin{equation}
   \label{small}
 \| \langle v\rangle^k f_{0}\|_{W^{1, \infty}} +   \|f_{0}\|_{W^{1, 1}} + \|f_{0}\|_{L^1_{x}L^\infty_{v}} + 
    \|\nabla_{x,v}f_{0}\|_{L^1_{x}L^\infty_{v}} \leq \eps_{0}.
\end{equation}
     Then  if $\eps_0>0$ is small enough, there exists a unique global solution of \eqref{VP}
      such that
   \begin{equation}
   \label{eq-thm} \|\rho (t) \|_{L^1} + \langle t\rangle \|\nabla_x \rho(t) \|_{L^1} + \langle t\rangle^d \|\rho (t) \|_{L^\infty}  +  \langle t\rangle^{d+1} \|\nabla_x \rho (t) \|_{L^\infty}
 \lesssim {  \eps_{0}} \log( 2 +t),
 \end{equation} 
for  all $t\ge 0$, with $\langle \cdot \rangle = \sqrt{1+|\cdot|^2}$.       
\end{thm}

 Theorem \ref{theomain} proves that the solution of \eqref{VP} enjoys the same decay properties as the free transport up to a logarithmic
 correction. 
  Note that we also establish the sharp higher decay of derivatives, still up to a logarithmic correction.  The same decay holds for $E = -\nabla_x (1-  \Delta_x)^{-1} \rho$ via standard elliptic estimates. 
 Improved decay of higher derivatives may actually be obtained with the same method introduced in this paper, but we shall not dwell on this point for the sake of conciseness.

 Note that we have not tried to optimize the needed regularity of the initial data in this statement,  ~\eqref{extra} means that we  ask slightly  more than Lipschitz regularity for the initial condition $f_0$. Since we can take $p$ arbitrarily large, this means that $\sigma$ can be taken arbitrarily close to one.
  We could actually take $\sigma = 1 + { 2 d \over p}$ by replacing  $W^{\sigma, p}$ by the Besov space $B^{\sigma}_{p, 1}$
 (this would change in particular  the local existence result of Proposition \ref{prop-LWP}).
  Note that  we do not ask for any smallness in~\eqref{extra}. This is used only in order to ensure that 
  a certain quantity (namely $\mathcal{N}(t)$ defined in \eqref{def-iterN}), that we shall use for a bootstrap argument, is continuous in time. 
  
  We could actually even only ask
   $\sigma >{2d \over p}$, so that loosely speaking, merely smallness in H\"older norm (instead of Lipschitz) would be needed, but this would require to replace the  $W^{1, \infty}$ and $W^{1, 1}$ estimates  for the density $\rho$ proved in the paper by more technical $\mathscr{C}^{0,\alpha}$ and $B^{\alpha}_{1, \infty}$
   estimates.
 
 \bigskip
 
 As a consequence of Theorem~\ref{theomain}, we obtain the aforemetioned scattering property for the solution to~\eqref{VP}. 
 \begin{coro}
 \label{coro1}
 With the same assumptions and notations as in Theorem~\ref{theomain}, there is $f_\infty \in W^{1,\infty}$
  given by 
   $$f_{\infty}(x,v) = f_{0}\left(x + Y_{\infty}(x,v), v+ W_{\infty}(x,v)\right) + \mu\left(v+ W_{\infty}(x,v)\right) - \mu(v)$$
   such that
 \begin{equation}
 \| f(t,x+tv, v) - f_\infty(x,v) \|_{L^\infty_{x,v}} \lesssim {\eps_0}\frac{\log(2+t)}{\langle t \rangle^{d-1}},
 \end{equation}
 for all $t \geq 0$. 
 Moreover, we also have that  $  \| Y_{\infty}\|_{L^\infty_{x,v}} +  \| W_{\infty}\|_{L^\infty_{x,v}} \lesssim \eps_{0}.$
 \end{coro}
 The proof which is an easy consequence of ~\eqref{eq-thm} will be  given in Section \ref{lastsection}. 
 \bigskip
 
 The remaining  of the paper is devoted to the proof of Theorem \ref{theomain} and is organized as follows. Linear estimates are derived in Section \ref{sec-lin}.
  This is the main ingredient in the proof: we shall establish in particular that the density  $\rho$ of  the solution of the   linearized 
   equation
$$
\left \{ 
\begin{aligned}
&\pa_t f + v\cdot \na_x f +  E \cdot \na_v \mu = 0 , \\
&E = -\nabla_x (1-  \Delta_x)^{-1} \rho, \qquad \rho(t,x)= \int_{\R^d} f (t,x,v)\, dv,\\
&f|_{t=0} = f_0
\end{aligned}
\right.
$$
    enjoys the same pointwise in time $L^p$ decay  as that of the  free transport equation (that corresponds to $\mu=0$), 
    up to a logarithmic correction. 
    
    The bootstrap argument allowing to get Theorem \ref{theomain}
 is introduced   in Section \ref{sec-nonlin}. Section \ref{sec-rechar} and \ref{sec-straightXV} are devoted to stability estimates for characteristics.  
In view of applying the linear estimates of Section~\ref{sec-lin}, several source terms need to be estimated.
The contribution from the initial data is studied in Section \ref{sec-bddata}, while the estimates on the terms due to the reaction term $-E\cdot \nabla_v \mu$ are established in Section \ref{sec-bdreaction}. 
  Note that in order to establish the higher decay of derivatives, we shall use a different, more straightforward change of variables
   than in \cite{Hwang}.
   The paper ends with a reminder of a few classical estimates related to the Littlewood-Paley decomposition.

\subsection{Notations}
 We use $\, \widehat \cdot \, $ for the "space" Fourier transform on $\mathbb{R}^d$ and $\, \widetilde  \cdot \, $ for the "space-time" Fourier transform on $\mathbb{R}^{d+1}$ with the convention:
\begin{align}
\label{fourier-space}
\widehat g (\xi) = \int_{\R^d} e^{- i x\cdot \xi} g(x) \; dx, 
\qquad \widetilde  h ( \tau, \xi) = \int_{\R}\int_{\R^d} e^{-i \tau t} e^{- i x\cdot \xi} h(t,x) \; dx dt .
\end{align}
Throughout the paper, functions depending on time are extended by zero for $t<0$. 

We shall use the homogeneous  Littlewood-Paley decomposition in 
$\mathbb{R}^n$ with $n=d$ or $d+1$. We write for $u \in \mathcal{S}'(\mathbb{R}^n)$, 
  $$ u = \sum_{q \in \mathbb{Z}} u_{q}$$
   where 
\begin{equation}
\label{defLP}
 \overline{ u_{q}}(\zeta) = \overline{u}(\zeta) \chi_{q}(\zeta), \quad \chi_{q}(\zeta)=  \chi({\zeta \over 2^q}), \qquad \overline{\cdot} \, = \, \widehat{\cdot} \text{ ~ or ~  } \widetilde{\cdot},
\end{equation}
    and $\chi \in [0,1]$ is a fixed smooth compactly supported function  in the annulus $ { 1 \over 4} \leq |\zeta | \leq 4$
      which is equal to one in the annulus  $ { 1 \over 2} \leq |\zeta | \leq 2$.
    The classical Bernstein Lemma is recalled in Lemma \ref{lemLp}.

\section{Linear estimates}\label{sec-lin}
In this section we study the linear equation
\begin{equation}\label{def-m}
\rho(t,x) = \int_0^t \int_{\R^d}  [\na_x (1 - \Delta_x)^{-1} \rho](s,x -(t-s)v) \cdot \na_v \mu(v) \, dv ds + S(t,x), \quad t \geq 0,
\end{equation}
with $S$ being a given source term. In what follows, we extend $\rho$ and $S$ by zero for $t<0 $ so that the equation \eqref{def-m}
is satisfied for $t \in \mathbb{R}$. The main result of this section is  the following. 
 
 \begin{thm}
 \label{thmvolterrasol}
 Assume that (H1) and (H2)
  are satisfied.  Then for all   $S \in L^1(\mathbb{R}, L^1(\mathbb{R}^d) \cap L^\infty(\mathbb{R}^d))$, there exists a unique 
 solution of \eqref{def-m} in $L^1_{loc}(\mathbb{R}, L^2(\mathbb{R}^d))$ that can be expressed in the following way:
 \begin{equation}
 \label{volterrasol} \rho= S + G*_{t,x} S,
 \end{equation}
 where the kernel $G(t,x)$ satisfies $G|_{t<0}=0$ and there exists $C>0$ such that the following uniform estimates hold:
 \begin{equation}
 \label{estimG}  \|G(t) \|_{L^1} \leq {  C \over 1 + t}, \qquad  \|G(t)\|_{L^\infty} \leq{ C \over t^{d-1 + \delta}+ t^{d+1}}, \quad \forall  t>0,
 \end{equation}
where $\delta \in (0, 1)$ can be chosen arbitrarily small. Furthermore, its spatial derivatives satisfy
 \begin{equation}
 \label{estimnablaG}  \|\nabla_x G(t) \|_{L^1} \leq {  C \over  t^2}, \qquad  \|\nabla_x G(t)\|_{L^\infty} \leq{ C \over t^{d+2}}, \quad \forall  t\geq 1.
 \end{equation}
  \end{thm} 
 As a corollary of Theorem \ref{thmvolterrasol}, we immediately obtain decay estimates for the solution of \eqref{def-m}.
 \begin{coro}
 \label{corodecayL}  Assume that (H1) and (H2)
  are satisfied. Then, there exists $M>0$ such that for all   $S \in L^1(\mathbb{R}, L^1(\mathbb{R}^d) \cap L^\infty(\mathbb{R}^d))$, the solution of \eqref{def-m} satisfies the estimates
 $$ 
 \begin{aligned}  \|\rho(t) \|_{L^1}  + t^d \|\rho(t) \|_{L^\infty}  & \leq  M \log  (1+t)  \|S\|_{Y_{t}^0} ,
 \\
 t\| \nabla \rho(t) \|_{L^1} + t^{d+1}\|\nabla \rho(t) \|_{L^\infty} &\leq M  {\log (1+t)} \|S\|_{Y_{t}^1} , 
 \end{aligned}$$
for $t\ge 1$, where the norms $Y_t^0, Y_t^1$ are defined by \begin{equation}
\label{def-Y0Y1}
\begin{aligned}
  \|S\|_{Y^0_{t}} &= \sup_{[0, t]} \left( \| S(s) \|_{L^1} + (1+ s)^d \|S(s) \|_{L^\infty}\right),  \\
   \|S\|_{Y^1_{t}} &=  \sup_{[0, t]} \left( \| S(s) \|_{L^1}
   + (1+ s)  \|\nabla S(s) \|_{L^1} +  (1+ s)^{d + 1}  \| \nabla S(s) \|_{L^\infty}\right). 
\end{aligned}
\end{equation}


  
 \end{coro}
 
 In Corollary \ref{corodecayL}, we state only large time estimates, since the estimates for $t \leq 1$ can be obtained in a straightforward way.
Note that derivatives decay at a $t^{-1}$ faster rate. In particular, Corollary \ref{corodecayL} immediately yields decay estimates for the linearized Vlasov-Poisson system of \eqref{VP} around $\mu(v)$, namely for
the system
 \begin{equation}
\label{VPL}
\left \{ 
\begin{aligned}
&\pa_t f + v\cdot \na_x f +  E \cdot \na_v \mu = 0 , \\
&E = -\nabla_x (1-  \Delta_x)^{-1} \rho, \qquad \rho(t,x)= \int_{\R^d} f (t,x,v)\, dv,\\
&f|_{t=0} = f_0.
\end{aligned}
\right.
\end{equation}
Indeed, using the method of characteristics, we obtain that $\rho$ solves \eqref{def-m} with
 $S(t,x)$ given by 
  $$ S(t,x)= \int_{\R^d} f_{0}(x-tv, v)\, dv. $$
  Assuming that $f_{0} \in L^1_{x,v}$ and $f_{0} \in L^1_{x}(L^\infty_{v})$, we have from the standard dispersive estimates
   for  free transport (see \cite{Bardos-Degond}) that
  $$ \|S(t) \|_{L^1}\leq \|f_{0}\|_{L^1}, \quad \|S(t) \|_{L^\infty} \leq { 1 \over t^d} \|f_{0}\|_{L^1_{x}L^\infty_{v}}.$$
  Similar estimates hold for derivatives:
  $$ \|\na S(t) \|_{L^1}\leq { 1 \over t^d}  \| \na_v f_{0}\|_{L^1}, \quad \|\na S(t) \|_{L^\infty} \leq { 1 \over t^{d+1}} \| \na_v f_{0}\|_{L^1_{x}L^\infty_{v}}.$$
  Therefore, we obtain from Corollary \ref{corodecayL}  that the pointwise behavior of the density of the linearized
  equation \eqref{VPL} is the same as the one of the free transport, up to a logarithmic loss. 
   
  Let us right away provide  the proof of Corollary \ref{corodecayL}.
  
  \begin{proof}[Proof of Corollary \ref{corodecayL}]
Using the representation \eqref{volterrasol} and the fact that $S$ and $G$ vanish for negative times
  we obtain
  $$ \rho(t)= S(t) + \int_{0}^t G(t-s) *_{x}  S(s) \, ds.$$
  Therefore, using \eqref{estimG}, we obtain
 $$ \begin{aligned} \|\rho(t)\|_{L^1} &\lesssim \|S(t) \|_{L^1} + \int_{0}^t { 1 \over 1 + (t-s)} \|S(s)\|_{L^1}\, ds
  \\  &\lesssim  \|S(t) \|_{L^1} +  { 1 \over 1+ t } \int_{0}^{t \over 2} \|S(s) \|_{L^1}\, ds +  \int_{t \over 2}^t  { 1 \over 1 + (t-s)}  \|S(s)\|_{L^1}\, ds
  \\  &  \lesssim  \left(1 + \int_{0}^{ t \over 2} { 1 \over 1 + s} \, ds\right) \sup_{[0, t]} \|S(s)\|_{L^1}.
      \end{aligned}$$
  In a similar way, we have 
 $$ \| \rho(t) \|_{L^\infty} \leq \|S(t) \|_{L^\infty} + \int_{0}^{t\over 2} \|G(t-s) \|_{L^\infty} \|S(s) \|_{L^1}\, ds
  + \int_{{t \over 2}}^t \|G(t-s) \|_{L^1} \|S(s) \|_{L^\infty} \, ds.$$
Therefore, using again \eqref{estimG}, we get that
 $$\begin{aligned}  \| \rho(t) \|_{L^\infty} 
 &\leq \|S(t) \|_{L^\infty}  + {1 \over t^{d + 1}}\int_{0}^{t \over 2}  \|S(s) \|_{L^1}\, ds
    + \int_{t \over 2}^t { 1 \over 1 + t-s} \|S(s) \|_{L^\infty} \, ds \\
   &  \lesssim  \|S(t) \|_{L^\infty} + {1 \over t^d} \sup_{[0, t/2]} \|S(s)\|_{L^1}
      + {1 \over t^d} \sup_{[t/2, t]} (1+s^d)\|S(s)\|_{L^\infty} \int_{0}^{ t \over 2}  { 1 \over 1 + s} \, ds
     \\&\lesssim t^{-d} \log  (1+t)  \|S\|_{Y_{t}^0} ,
    \end{aligned} $$
upon recalling the notation \eqref{def-Y0Y1}. This proves the desired estimates for $\rho(t)$. Similarly, we compute
   $$ t\|\nabla \rho(t)\|_{L^p} \lesssim t\|\nabla S(t) \|_{L^p} +  t  \int_{0}^{t \over 2} \|\nabla G(t-s) \|_{L^p} \|S(s) \|_{L^1}\, ds + t  \int_{t \over 2}^t  \|G(t-s)\|_{L^1} \|\nabla S(s)\|_{L^p}\;ds$$
   for $p =1$ and $p=\infty$.  
By using \eqref{estimG} and \eqref{estimnablaG}, the estimates for derivatives follow.  
%
  \end{proof}
  
 Before giving the proof of  Theorem \ref{thmvolterrasol} it will be useful to establish some properties
  of the kernel of the  integral equation \eqref{def-m}.
  
   Let us  set  $ w= x- (t-s)v$ and  then  integrate by parts to get the equivalent formulation
  \begin{equation}
  \label{volterrabis}
  \rho(t,x) =  \int_0^t \int_{\R^d}  \rho(s,w)  \left( \left( 1 - {1 \over (t-s)^2} \Delta_{v} \right)^{-1} \Delta_v \mu \right)\left(
   {x-w \over t-s}\right) { 1 \over (t-s)^{d+1} }\, dw ds + S(t,x), \quad t \geq 0.
 \end{equation}
 Since  $\rho$ and $S$ by zero for $t<0 $,  the equation \eqref{def-m}
  is satisfied for $t \in \mathbb{R}$ and can be  rewriten  as the convolution equation
  \begin{equation}
  \label{lin2}
  \rho(t,x) = ( K *_{t,x} \rho) (t,x) + S(t,x), \quad t \in \mathbb{R},
  \end{equation}
  where the kernel $K$ is given by 
  $$ K(t,x)=  { 1 \over t^{d+1}} \left( \left( 1 - {1 \over t^2} \Delta_{v} \right)^{-1} \Delta_v \mu \right)\left(
   {x \over t}\right) \mathrm{1}_{t>0}.$$
   Note that we have
   $$ ( K *_{t,x} \rho) (t,x)= \int_{\R}\int_{\mathbb{R}^{d}} K(t-s, x-w) \rho(s, w) dw ds
   = \int_{0}^t \left(K(t-s, \cdot) *_{x} \rho(s, \cdot)\right)(x) \, ds,$$
   where we use the notation $*_{t,x}$ for the space-time convolution and $*_{x}$ for the space convolution.
   
   We have the following properties for the kernel $K$:
   \begin{lem}
   \label{lemK}
   Assuming (H1),  there exists $C>0$ such that the following estimates hold:
    $$ \|K(t) \|_{L^1} \leq {  C \over 1 + t}, \quad  \|K(t)\|_{L^\infty} \leq{ C \over t^{d}(1+ t)}, \quad \forall  t>0.$$
    \end{lem} 
Note that we get in particular  from this lemma that $K \in L^1_{loc}(\mathbb{R}, L^1(\mathbb{R}^d))$ so that its Fourier
transform on $\mathbb{R}^{d+1}$ is well defined (at least as a tempered distribution).
We shall not use explicitly these precise properties of $K$ besides the fact that its Fourier transform makes sense. 
  %
  %
  \begin{proof}
  We clearly get that
$$
\| K(t, \cdot) \|_{L^\infty} \leq { C \over  t^{d+1}} \| F_{t} \|_{L^\infty},
$$
  and by change of variables, that
   $$ \| K(t, \cdot) \|_{L^1} \leq { C \over  t} \| F_{t} \|_{L^1},$$
   where 
 $$   F_{t} (v)= \left( \left( 1 - {1 \over t^2} \Delta_{v} \right)^{-1} \Delta_v \mu  \right) (v),$$
 so that it only remains to estimate the $L^1$ and $L^\infty$ norms of $F_{t}.$
    We use the homogeneous Littlewood-Paley decomposition and the Bernstein inequality~\eqref{Bernstein1} to get that for every
     $p \in [1, +\infty]$, 
  $$
   \| F_{t}\|_{L^p} \lesssim \sum_{q\in \mathbb{Z}} {  2^{2q} \over 1 +  {2^{2 q} \over t^2}} \|\mu_{q} \|_{L^p}.$$
   We can thus obtain from  that uniformly for $t >0$,
  $$  \| F_{t}\|_{L^p} 
   \lesssim \sum_{q \leq 0}  2^{2q} \|\mu_{q} \|_{L^p}  + \sum_{q \geq 0}  2^{-q} \|\mu_{q}\|_{W^{3,p}}
    \lesssim  \|\mu \|_{W^{3,p}}.
    $$
    This yields the estimates for $t \geq 1$.
     It remains to improve the estimates  for $t \in (0, 1].$
   This time,  we write
  $$     \| F_{t}\|_{L^p} \lesssim  \sum_{q \leq 0}  t 2^{q} \|\mu_{q} \|_{L^p}
   +  \sum_{q \geq 0}   t^2 2^{-q} \|\mu_{q}\|_{W^{1,p}} \lesssim t \|\mu\|_{W^{1,p}},
   $$
hence concluding the proof.
     
  \end{proof}

  \subsection*{Proof of Theorem \ref{thmvolterrasol}}
    We now give the proof of Theorem \ref{thmvolterrasol} which we shall split in several steps.
  As already justified in \cite{MV,BMM2}, 
  we can express the solution
   of \eqref{def-m} through its space-time Fourier transform by 
   $$ \widetilde {\rho}(\tau,\xi) = { 1 \over 1-  \widetilde {K}(\tau, \xi) } \widetilde {S}(\tau, \xi)$$
in which $\widetilde {K}(\tau, \xi)$ is given by
  \begin{equation}
  \label{Ktildedef} \widetilde  K (\tau, \xi)=    \int_{0}^{+\infty} e^{-i \tau t} {  i  \xi \over 1 + |\xi|^2} \cdot \widehat{\nabla_{v} \mu}
  (t \xi) \, dt.
  \end{equation}  
    The fact that the inverse Fourier transform of  ${ 1 \over 1-  \widetilde{K}(\tau, \xi)}  \widetilde {S}(\tau, \xi)$
     vanishes for $t<0$ comes from a Paley-Wiener type argument and uses the fact that
      $  1-  \widetilde {K}(z, \xi)$ does not cancel in the half-plane $\Im\, z \leq 0$, which is precisely the Penrose stability condition (H2). 
      
      We can then write 
   $$   \widetilde {\rho}(\tau,\xi) =  \widetilde {S}(\tau, \xi)  + { \widetilde  K(\tau, \xi) \over 1-  \widetilde {K}(\tau, \xi) } \widetilde {S}(\tau, \xi)$$
   and hence the expression \eqref{volterrasol} follows by setting
    \begin{equation}
    \label{Gdef} G(t,x)= \mathcal{F}^{-1}_{(\tau, \xi)\rightarrow (t,x)} \Big({ \widetilde  K(\tau, \xi) \over 1-  \widetilde {K}(\tau, \xi) } \Big).
    \end{equation}
It thus remains to check the claimed properties of $G$.
 
 At first, we observe that for $\xi \neq 0$,  ${\widetilde  K( z, \xi) \over 1 -  \widetilde  K(z,\xi)}$ is a holomorphic function in $\Im \, z < 0$ thanks to (H1) and (H2):
 \begin{itemize}
 \item the Penrose condition (H2) ensures that $ 1 -  \widetilde  K(z,\xi)$ is away from $0$,
 \item and (H1) entails that  $ \widehat{\nabla_{v} \mu}$ and its derivatives are  decreasing sufficiently fast, so that one can apply Lebesgue dominated convergence theorem.
 \end{itemize}

Now, note that thanks to (H1) and (H2),  ${\widetilde  K( z, \xi) \over 1 -  \widetilde  K(z,\xi)}$ is uniformly bounded 
    in  $\Im \, z \leq  0$.  Indeed, by (H1), $ {\nabla_{v} \mu} \in W^{2,1}$ and therefore we get 
   \begin{equation}
   \label{Kborne1}
   \forall \tau \in \R, \gamma\geq0, \quad |\widetilde {K}(\tau-i \gamma, \xi)|
     \lesssim  \int_{0}^{+ \infty} e^{-\gamma t} {| \xi| \over (1 + t |\xi|)^2} \, dt \lesssim  
      \int_{0}^{+ \infty} {ds \over (1 + s)^2}.
      \end{equation}
In addition, by integration by parts, we can compute 
   \begin{align*}
   \widetilde  K (\tau, \xi) &=    \int_{0}^{+\infty} \frac{(1-\pa_t^2)(e^{-i \tau t}) }{1+\tau^2} {  i  \xi \over 1 + |\xi|^2} \cdot \widehat{\nabla_{v} \mu}
  (t \xi) \, dt \\
  &=   \int_{0}^{+\infty} \frac{e^{-i \tau t} }{1+\tau^2}{  i  \xi \over 1 + |\xi|^2} \cdot (1-\pa_t^2) \widehat{\nabla_{v} \mu}
  (t \xi) \, dt  - \frac{1}{1+\tau^2}{  i \xi \over 1 + |\xi|^2} \xi \cdot \na_\xi [\widehat{\nabla_{v} \mu}](0).
   \end{align*}
   Note that we have used that  $\widehat{\nabla_{v} \mu}(0)=0$, which ensures that a boundary term cancel. By (H1), $\langle v\rangle^2 \na_v \mu \in W^{2,1}$.
    This yields, arguing similarly as for~\eqref{Kborne1}, for any fixed $\xi \neq 0$,
   \begin{equation*}
| \widetilde {K}(\tau, \xi)| \lesssim { 1 \over 1 + \tau^2}.
\end{equation*}
    Since a similar expression holds for the derivatives $\partial_{\tau}^\alpha  \widetilde {K}(\tau, \xi)$ and by (H1), $\langle v\rangle^2 \na_v \mu \in W^{4,1}$, for any fixed $\xi \neq 0$, we have 
  \begin{equation}
  \label{eq-dalphaK}
|\partial_{\tau}^\alpha  \widetilde {K}(\tau, \xi)| \lesssim { 1 \over 1 + \tau^2},  \quad \alpha  \leq 2,
\end{equation}
uniformly in $\tau$. 
   
%
%

    This entails in particular that ${\widetilde {K}(\tau, \xi) \over 1- \widetilde  {K}(\tau ,\xi)}$ and its inverse Fourier transform in time
    have moderate decrease.  Therefore, thanks to an adequate version of  the Paley-Wiener theorem
    (see Theorem 3.5 in \cite{Stein}), we get that
    $ \mathcal{F}^{-1}_{\tau\rightarrow t}\left( \frac{\widetilde {K}}{1-\widetilde{K}}\right)$ vanishes for $t <0$ for every $\xi$. We conclude that $G$ vanishes
     for $t<0$.
     
     \bigskip
     
   It remains to prove the pointwise decay estimates \eqref{estimG}-\eqref{estimnablaG}. 
We need the following properties of $\widetilde {K}(\tau, \xi)$.
  \begin{lem}
  \label{lemhomogene}
    We can write 
      \begin{equation}
     \label{K1def} \widetilde {K}(\tau, \xi)=    { 1 \over 1 + |\xi |^2} \widetilde  K^{h,1}(\tau, \xi)
     \end{equation}
     where $ \widetilde  K^{h,1}(\tau, \xi)$ is positively homogeneous of degree zero and $\widetilde  K^{h, 1} \in \mathscr{C}^{2d+3}(\mathbb{R}^{d+1} \backslash\{0\})$.
   Moreover, there exists $C>0$ such that 
   \begin{equation}
   \label{estKh1}| \partial_{\tau}^\alpha \partial_{\xi}^\beta \widetilde  K^{h,1}(\tau, \xi) |
    \leq C, \qquad \forall (\alpha, \beta), \, | \alpha| + |\beta | \leq  2 d+ 3,\quad \, \forall (\tau, \xi) \in \mathbb{S}^d,
    \end{equation}
     where $\mathbb{S}^d$ is the unit sphere of $\mathbb{R}^{d+1}$.  
  \end{lem}

  \begin{proof}
 From the definition \eqref{Ktildedef}, we have 
  $$  \widetilde  K^{h,1}(\tau, \xi)=  \int_{0}^{+\infty} e^{-i \tau t}   i  \xi \cdot \widehat{\nabla_{v} \mu}
  (t \xi) \, dt .$$
 We observe that for $\lambda >0$,
    $$\widetilde  K^{h,1}(\lambda \tau, \lambda \xi)= \int_{0}^{+ \infty} e^{- i \lambda \tau t} i \lambda \xi \cdot \widehat{\nabla_{v} \mu}
    (t \lambda \xi)\, dt= \widetilde  K^{h,1}(\tau, \xi),$$
upon using the change of variable $s= \lambda t$.  Note that by using again \eqref{Kborne1}, we have
         $$ \sup_{(\tau, \xi) \in \mathbb{R}^{d+1}} |\widetilde  K^{h,1}(\tau, \xi)|<+\infty.$$
  To estimate the derivatives on the sphere, we first handle the case when $|\xi| \geq {1 \over 2}.$ 
Thanks to (H1), $\langle v \rangle^{2d+3} \na_v \mu \in W^{2d+5,1}$ and consequently, we have
   $$ | \partial_{\tau}^\alpha \partial_{\xi}^\beta \widetilde  K^{h,1}(\tau, \xi) | \lesssim 
     \int_{0}^{+ \infty}    { \langle t \rangle^{|\alpha| + | \beta|} \over (1 + t|\xi|)^N}\, dt$$
for $N=2d+5$,
     and therefore, for $|\xi| \geq {1\over 2}$ and $| \alpha| + |\beta | \leq  2 d+ 3$, we have
     \begin{equation}\label{bd-X1}   | \partial_{\tau}^\alpha \partial_{\xi}^\beta \widetilde  K^{h,1}(\tau, \xi) | \lesssim  1, \qquad  (\tau, \xi) \in \mathbb{S}^d, \, 
\quad      | \xi | \geq { 1 \over 2}.\end{equation}

Let us next consider the case when $|\xi| \leq { 1 \over 2}$, in which we make use of the fact that $\tau$ is bounded below away from zero, recalling $(\tau, \xi) \in \mathbb{S}^d$. Integrating by parts, we get for every $n \geq 2$
   \begin{equation}
   \label{Kxismall}
    \widetilde {K}^{h,1}(\tau, \xi)=  \sum_{k= 2}^n
     { 1 \over (i \tau)^k}  \mathscr{P}_{k}(\xi)   +   { 1 \over (i\tau)^n} R_{n}(\tau, \xi)
    \end{equation} 
     where
     $$\begin{aligned} \mathscr{P}_{k}(\xi)&= i \xi \cdot (D_{\xi}^{k-1} \widehat{\nabla_{v} \mu})(0): \xi^{\otimes k-1}, \\
       R_{n}(\tau, \xi) &= \int_{0}^{+ \infty} e^{-i \tau t} r_{n}(t,\xi)\, dt, \quad r_{n}(t,\xi)=
       i \xi\cdot (D_{\xi}^n \widehat{\nabla_{v} \mu}) (t\xi)):\xi^{\otimes n},
       \end{aligned}
       $$
        with the definition
  $$ \xi \cdot (D_{\xi}^k \widehat{\nabla_{v} \mu})(\zeta): \xi^{\otimes k}= \sum_{j_{0}, j_{1}, \cdots j_{k}}
  \xi_{j_{0}}  \xi_{j_{1}} \cdots \xi_{j_{k}}{\partial^k \over
    \partial_{\xi_{j_{1}} } \cdots \partial_{\xi_{j_{k}}}   } \widehat{\partial_{v_{j_{0}}} \mu}(\zeta) .$$
     Note that $\mathscr{P}_{k}$ is a homogeneous polynomial of degree $k$.  Thanks to (H1), $\langle v\rangle^{2d+3} \na_v \mu \in W^{2d+5,1}$  and thus we have for all $n \leq 2d+3$,
   $$  | r_{n}(t, \xi) |
    \lesssim { |\xi|^{n+1} \over (1 + t | \xi |)^N},$$
for $N=2d+5$. More generally, using $\langle v\rangle^{4d+6} \na_v \mu \in W^{2d+5,1}$, we have for all $n \leq 2d+3$ and $|\beta|\le 2d+3$, 
    $$  |\partial_{\xi}^\beta  r_{n}(t, \xi) | \leq { |\xi|^{n+1- |\beta |} \over (1 + t |\xi|)^{N-|\beta|}}.$$
   Consequently, applying derivatives  to the expansion  \eqref{Kxismall}, we get for $|\xi| \leq { 1 \over 2}$, 
    $$  | \partial_{\tau}^\alpha \partial_{\xi}^\beta \widetilde  K^{h,1}(\tau, \xi) | \lesssim 
     1 +  \int_{0}^{+ \infty} { t^{|\alpha|} |\xi|^{n +1- | \beta|} \over (1+ t |\xi|)^{N-|\beta|} }\, dt
      \lesssim 1 +  \int_{0}^{+ \infty} { s^{|\alpha|} |\xi|^{n - | \beta| - |\alpha |} \over (1+ s)^{N-|\beta|} }\, ds.$$
 Thus, we can  fix  $n=2d+3$.  We get for $| \alpha| + |\beta | \leq  2 d+ 3$,
   $$  | \partial_{\tau}^\alpha \partial_{\xi}^\beta \widetilde  K^{h,1}(\tau, \xi) | \lesssim 
     1, \qquad  (\tau, \xi) \in \mathbb{S}^d, \, \quad
      | \xi | \leq { 1 \over 2}.$$        
      This, together with \eqref{bd-X1}, concludes the proof.
  \end{proof} 
  
  \begin{rem}
  \label{remxismall}
  We can also express the expansion \eqref{Kxismall}  for $n= 2$ in a slightly different way.
   We write 
  \begin{equation}
  \label{K1defbis}  \widetilde {K}^{h,1}(\tau, \xi)=  { 1 \over (i \tau)^2} \mathscr{P}_{2}(\xi)+ { \xi^{\otimes 2} \over (i\tau)^2} : \widetilde {K}^{h,2}(\tau, \xi)
  \end{equation}
   where $\mathscr{P}_{2}$ is a homogeneous polynomial of degree two in $ \xi$ and
   $$ \widetilde  K^{h,2} (\tau, \xi)= \int_{0}^{+ \infty}   e^{-i \tau t} i \xi \cdot (D_{\xi}^2 \widehat{\nabla_{v}\mu})(t\xi) \, dt.$$
    From the same arguments as in the proof of Lemma~\ref{lemhomogene}, we get that $\widetilde {K}^{h,2}(\tau, \xi)$ is homogeneous of degree zero on $\mathbb{R}^{d+1}$
     and that 
     \begin{equation}
     \label{estKh2}
     | \partial_{\tau}^\alpha \partial_{\xi}^\beta \widetilde  K^{h,2}(\tau, \xi) |
    \leq C, \qquad \forall (\alpha, \beta), \, | \alpha| + |\beta | \leq  2 d+ 3, \, \quad\forall (\tau, \xi) \in \mathbb{S}^d.
    \end{equation}
   \end{rem}
  
Combining Remark \ref{remxismall} and Lemma \ref{lemhomogene}, we obtain 
  \begin{coro}
  \label{corK}
    We can write the expansion
    $$ \widetilde {K}(\tau, \xi) = { 1 \over 1 + |\xi|^2 + \tau^2} \left( \widetilde {K}^{h,1}(\tau, \xi)  - {\mathscr{P}_{2}(\xi) \over { 1 + |\xi|^2}}  - { \xi^{\otimes 2} \over  1+ |\xi|^2} : \widetilde {K}^{h,2}(\tau, \xi) \right)$$
    where $\widetilde {K}^{h,1}$,  $\widetilde {K}^{h, 2}$ are positively homogeneous of degree zero and satisfy the estimates
     \eqref{estKh1}, \eqref{estKh2} and $\mathscr{P}_{2}(\xi)$ is a homogeneous polynomial of degree $2$.
  \end{coro}
  \begin{proof}
The corollary follows from a  combination of \eqref{K1def} and \eqref{K1defbis}. 
  \end{proof}
  
  Let us now use the properties of $\widetilde K(\tau,\xi)$ to derive the pointwise decay estimates for $G(t,x)$, recalling \eqref{Gdef}. 
We first use the homogeneous Littlewood-Paley decomposition of $\mathbb{R}^{d+1}$
to decompose $G(t,x)$, yielding
    $$ G(t,x)= \sum_{q \in \mathbb{Z}} G_{q}(t,x), 
    $$
   recalling~\eqref{defLP}.
The general strategy will consist of treating differently the contribution of high and low frequencies. 
   We first deal with high frequencies: 
   \begin{lem}
   \label{lemhigh}
    There exist $A \geq 1$ and $C>0$ such that for every $\delta \in (0, 1]$ and every  $q$ with $2^q \geq A$, we have the estimates
   \begin{equation}
   \label{Gqhigh} \|G_{q}(t) \|_{L^1} \leq C { 2^{q( 1 + \delta)}  \over 1 + 2^{2q}} {1 \over (1 +  2^q|t|)^N}, \qquad
    \|G_{q}(t)\|_{L^\infty} \leq C { 2^{q(d+1 + \delta)} \over 1 + 2^{2q}} { 1 \over (1 + 2^q|t|)^N}, \quad \forall t \in \R,
    \end{equation}
for $N = d+3$. 
   \end{lem}
   \begin{proof}
    We first observe that we can rewrite \eqref{Gdef} under the form
     $$ \widetilde {G}= \widetilde  K + \widetilde  K \widetilde  G.$$
Let $\varphi$ be a smooth cut-off function that is supported in an annulus slightly larger than that of $\chi$, so that $\varphi= 1$ on the support of $\chi$. 
Setting
    $$   \widetilde {\mathcal{K}}_{q}= \varphi\left(\frac{(\tau,\xi)}{2^q}\right) \widetilde {K},$$
we hence get the convolution equation
     \begin{equation}
     \label{convolGq}G_{q}=  K_{q} + {\mathcal{K}}_{q} *_{t,x} G_{q} .
     \end{equation}
          
Let us first bound $K_q$. We shall prove that for all $q>1$,
         \begin{equation}
    \label{Kq1} \|{K}_{q}(t)\|_{L^1_x}
      \lesssim  { 2^{q \delta} \over 1 + 2^{2q}}  { 2^q \over (1 + 2^q |t|)^N},
      \end{equation}
for $N = d+3$. 
Using  the expansion in Corollary \ref{corK},
     we write
     $$
     \widetilde {K}_q = \widetilde {K}_{q,1} + \widetilde {K}_{q,2}, 
     $$
     with
     $$
     \begin{aligned}
  \widetilde {K}_{q,1}&=   { 1 \over 1 + |\xi|^2 + \tau^2} \left( \widetilde {K}^{h,1}(\tau, \xi)    - { \xi^{\otimes 2} \over  1+ |\xi|^2} : \widetilde {K}^{h,2}(\tau, \xi) \right) \chi_q(\tau,\xi), \\
   \widetilde {K}_{q,2}&=    { -1 \over 1 + |\xi|^2 + \tau^2}{\mathscr{P}_{2}(\xi) \over { 1 + |\xi|^2}} \chi_q(\tau,\xi).
   \end{aligned}
  $$

We first check \eqref{Kq1} for $K_{q,1}$.  By a scaling argument, we can write
   $$ {K}_{q,1}(t, x) = 2^{q(d+1)} k_{q,1}(2^qt,  2^q x)$$ where 
   $$  k_{q,1}(T, X)=  \mathcal{F}^{-1}_{(\tau, \xi) \rightarrow (T,X)}
     \widetilde  K_{q,1}(2^q \tau, 2^q \xi).$$ 
We claim that
     \begin{equation}
\label{eq-kq1}
|\partial_{\tau}^\alpha \partial_{\xi}^\beta \widetilde  k_{q,1} (\tau,\xi)| \lesssim \frac{1}{1+2^{2q}},
\end{equation}
  uniformly in $q$ for $q>1$ for $| \alpha| + |\beta | \leq N + d$.
This  yields  \eqref{Kq1} for $K_{q,1}$. Indeed, from \eqref{eq-kq1}, we obtain from taking integration by parts
   $$  
   |k_{q,1}(T, X)| \lesssim  \frac{1}{1+2^{2q}}  { 1 \over (1 + |T| + |X|)^{ N+d}},
   $$
which in turn implies   
$$
 \|k_{q,1}(T) \|_{L^1_X} \leq \frac{1}{1+2^{2q}}  {1 \over (1 + |T|)^N}.
 $$
Therefore, by a change of variables, we obtain
$$
\|{K}_{q,1}(t)\|_{L^1_x}
      \lesssim  { 1 \over 1 + 2^{2q}}  { 2^q \over (1 + 2^q |t|)^N}.
$$  
We now prove the claim \eqref{eq-kq1}. By  homogeneity of $ \widetilde {K}^{h,1}$ and $\widetilde {K}^{h,2}$, we can write
    $$   \widetilde  k_{q,1}(\tau, \xi)=    { 1 \over 1 + | 2^q \xi|^2 + (2^q \tau)^2} \left( \widetilde {K}^{h,1}(\tau, \xi)  - { \xi^{\otimes 2} \over  2^{-2q}+ |\xi|^2} : \widetilde {K}^{h,2}(\tau, \xi) \right) \chi(\tau, \xi).$$
     Since $\widetilde {K}^{h,1}$ and $\widetilde {K}^{h,2}$ are smooth homogeneous  functions of
degree
      zero that satisfy~\eqref{estKh1} and~\eqref{estKh2}, we have
      $$
    | \partial_{\tau, \xi}^\gamma \widetilde {K}^{h,j} (\tau,\xi)| \lesssim |(\tau,\xi)|^{-|\gamma|}\lesssim 1,
      $$
since on support of $\chi$, $|\xi|+ |\tau|$ is bounded from below by a strictly positive number. On the other hand,
   $$
     \left| \partial_{\tau, \xi}^\gamma \left({ 1 \over 1 + | 2^q \xi|^2 + (2^q \tau)^2}  \right) \right|
   \lesssim { 2^{q |\gamma |} \over (1 + 2^{2q}( |\xi |^2 + \tau^2))^{1 + |\gamma|/2}}
    \lesssim  \frac{1}{1+2^{2q}}
      $$
on the support of $\chi$.  We thus deduce the estimate \eqref{eq-kq1}, and hence \eqref{Kq1} for $K_{q,1}$. 

For what concerns $K_{q,2}$, using~\eqref{riesz2}, for any $\delta \in (0,1]$, we have
 $$
 \| K_{q,2}(t)\|_{L^1_x} \lesssim  2^{q\delta} \left\| \mathcal{F}^{-1}_{(\tau, \xi) \rightarrow (t,x) } \left({ 1 \over 1 + |\xi|^2 + \tau^2} \chi_q(\tau,\xi)\right) \right\|_{L^1_x},
 $$
 and arguing as for $K_{q,1}$ with a scaling argument, we deduce~\eqref{Kq1}.

      

 Moreover, exactly as above, we obtain for all $q>1$,
      \begin{equation}
    \label{Kq1bis} \|{\mathcal{K}}_{q}(t)\|_{L^1_x}
      \lesssim  { 2^{q \delta} \over 1 + 2^{2q}}  { 2^q \over (1 + 2^q |t|)^N} .
      \end{equation}

      \bigskip

   Therefore, rewriting \eqref{convolGq} as
   \begin{equation}
   \label{convolGq2}
    G_{q}(t,x)= K_{q}(t,x) + \int_{\mathbb{R}} \left({\mathcal{K}}_{q}(t-s, \cdot) *_{x} G_{q}(s, \cdot) \right)(x) \, ds,
   \end{equation}
and taking the $L^1$ norm in $x$ and by using \eqref{Kq1} and~\eqref{Kq1bis}, we thus obtain that
 $$  \|G_{q}(t) \|_{L^1} \lesssim  { 2^{q \delta} \over 1 + 2^{2q}}  { 2^q \over (1 + 2^q |t|)^N}
  + \int_{\mathbb{R}}  { 2^{q \delta} \over 1 + 2^{2q}}  { 2^q \over (1 + 2^q |t-s|)^N} \|G_{q}(s) \|_{L^1}\, ds.$$
  Let us set for all $T\in \R$
  $$ |||G_q|||_{1,T} =  (1 + |T|)^N \|G_{q} ( { T \over 2^q})\|_{L^1}.$$ 
  We deduce after a change of variables that
\begin{equation}\label{bdGGG}  \begin{aligned} |||G_q|||_{1,T} & \lesssim  { 2^q  2^{q \delta} \over 1 + 2^{2q} }  +
    { 2^{q \delta} \over 1 + 2^{2q} }  \int_{\mathbb{R} } { ( 1 + |T|)^N \over  (1 +|T-S|)^N } {1 \over (1 + |S|)^N} |||G_q|||_{1,S}  \, dS
    \\ &\lesssim { 2^q 2^{q \delta} \over 1 + 2^{2q}}   +
    { 2^{q \delta} \over 1 + 2^{2q}}  \sup_{S\in \mathbb{R}}  |||G_q|||_{1,S}  ,
    \end{aligned}\end{equation}
    where we have used that for $N>1$, 
   $$ \sup_{T \in \mathbb{R}}  \int_{\mathbb{R} } {  ( 1 + |T|)^N\over  (1 +|T-S|)^N } {1\over (1 + |S|)^N} \, dS <+\infty.$$
  Consequently, after taking the sup in $T$, we can find $A>1$ sufficiently large 
  such that for all $q$ satisfying $2^q>A$, the last term on the right of \eqref{bdGGG} is absorbed into the left, yielding 
  $$  \begin{aligned} 
  \sup_{T}|||G_q|||_{1,T} 
  \lesssim { 2^q 2^{q \delta} \over 1 + 2^{2q}}.     
  \end{aligned}$$
This proves the $L^1$ estimate in \eqref{Gqhigh}. 
  
  It remains to estimate the $L^\infty$ norm. The proof  follows the same lines. 
Arguing as for~\eqref{Kq1}, we obtain    
$$\|{K}_q(t)\|_{L^\infty(\mathbb{R}^d)} + \|\mathcal{K}_q(t)\|_{L^\infty(\mathbb{R}^d)}
      \lesssim  { 2^{q \delta} \over 1 + 2^{2q}}  { 2^{q(d+1)} \over (1 + 2^q |t|)^N},
   $$
   for any $\delta  \in (0,1]$.
We then get by using   \eqref{convolGq} and \eqref{Kq1}  that
$$\begin{aligned} \|G_{q}(t) \|_{L^\infty} &\lesssim 
 { 2^{q \delta} \over 1 + 2^{2q}}  { 2^{q(d+1)} \over (1 + 2^q |t|)^N}+
  \int_{\mathbb{R}} \|{\mathcal{K}}_{q}(t-s)\|_{L^1} \|G_{q}(s) \|_{L^\infty} ds
  \\  &\lesssim   { 2^{q \delta} \over 1 + 2^{2q}}  { 2^{q(d+1)} \over (1 + 2^q |t|)^N}+ \int_{\mathbb{R}}
 { 2^{q \delta} \over 1 + 2^{2q}}  { 2^q \over (1 + 2^q |t-s|)^N} \|G_{q}(s) \|_{L^\infty}\, ds.
\end{aligned}$$
   We then conclude as before by setting
   $$  |||G_q|||_{\infty,T}  =  (1 + |T|)^N \|G_{q} ( { T \over 2^q})\|_{L^\infty},$$ 
   that 
   $$ \sup_{T} |||G_q|||_{\infty,T} \lesssim   { 2^{q \delta} 2^{q(d+1)} \over 1 + 2^{2q}}  + 
  { 2^{q \delta} \over 1 + 2^{2q}}  \sup_{S} |||G_q|||_{\infty,S}.$$ 
  The estimate \eqref{Gqhigh} for the $L^\infty$ norm thus follows by choosing $A$ sufficiently large. 
       \end{proof}

From now on, $A>1$ is fixed  and 
there remains to estimate $G_{q}$ for $ 2^q \leq A$. This is the content of the next lemma.
 \begin{lem}
   \label{lemlow}
 For $A> 1$,   there exists  $C>0$ such that for every $q\in \mathbb{Z}$ with $2^q \leq A$, we have the estimate
   \begin{equation}
   \label{Gqlow} \|G_{q}(t) \|_{L^1} \leq C  { 2^{q} \over (1 +  2^q|t|)^N}, \qquad
    \|G_{q}(t)\|_{L^\infty} \leq C { 2^{q(d+1)}  \over (1 + 2^q|t|)^N}, \quad \forall t \in \R,
    \end{equation}
for $N = d+3$. 
\end{lem}
 \begin{proof}
  We use directly the expression \eqref{Gdef} and  argue as in the proof of Lemma~\ref{lemhigh}.
  By a scaling argument we can write
   that 
   $$ G_{q}(t, x) = 2^{q(d+1)} g_{q}(2^qt,  2^q x)$$ where 
   $$  g_{q}(T, X)=  \mathcal{F}^{-1}_{(\tau, \xi) \rightarrow (T,X)}
       {  \widetilde  K(2^q \tau, 2^q \xi)   \over  1 -   \widetilde  K(2^q\tau, 2^q\xi) } \chi(\tau, \xi).$$ 
     To get the result, it is sufficient to use  prove that
     $\partial_{\tau}^\alpha \partial_{\xi}^\beta \widetilde g_{q} $ is bounded on the support of $\chi$
     uniformly in $q$ for $2^q \leq A$ for $| \alpha| + |\beta | \leq N + d$.
   From the Penrose stability condition (H2), such an estimate for $\widetilde g_{q}$ follows from  a similar  one
  bearing  on $$\widetilde k_{q}(\tau, \xi) = \widetilde K( 2^q \tau, 2^q \xi) \chi(\tau,\xi).$$
    From Corollary \ref{corK}, we can write by homogeneity  that
    $$ \widetilde  k_{q}(\tau, \xi)= { 1 \over 1 + | 2^q \xi|^2 + (2^q \tau)^2} \left( \widetilde {K}^{h,1}(\tau, \xi)  - {\mathscr{P}_{2}(\xi) \over {{2^{-2q}} + |\xi|^2}}  - { \xi^{\otimes 2} \over  {2^{-2q}}+ |\xi|^2} : \widetilde {K}^{h,2}(\tau, \xi) \right) \chi(\tau, \xi).$$
 Since on the support of $\chi$, $|\xi|+ |\tau|$ is bounded  below by a  positive number, 
  we have that  uniformly in $q$,
  $$\left| \partial_{\tau, \xi}^\gamma \left({ 1 \over 1 + | 2^q \xi|^2 + (2^q \tau)^2}  \right) \right|
   \lesssim { 2^{q |\gamma |} \over (1 + 2^{2q}( |\xi |^2 + \tau^2))^{1+|\gamma|/2 }}
    \lesssim  1.$$
     We also observe that on the support of $\chi$,  $ \xi$ belongs to a ball  so that we  have
     $$\left | \partial_{\xi}^\alpha \left(  {\mathscr{P}_{2}(\xi) \over { { 2^{-2q}} + |\xi|^2}} \right)\right | \lesssim{ 2^{2q |\alpha|}}\lesssim 1.$$
     This is for the control of this term that we use that we are in the low frequency regime $2^q \leq A$.
      Since  $\widetilde {K}^{h,1}(\tau, \xi)$ and  $\widetilde {K}^{h,2}(\tau, \xi)$ satisfy \eqref{estKh1}, \eqref{estKh2}, 
      the uniform estimate for  $\partial_{\tau}^\alpha \partial_{\xi}^\beta \widetilde k_{q} $  follows.       
  \end{proof}  
  
   \subsection*{End of the Proof of Theorem \ref{thmvolterrasol}}
   
 We shall combine Lemma \ref{lemhigh} and Lemma \ref{lemlow}. Let $A$ be given 
   by Lemma \ref{lemhigh}.

  We have
 \begin{equation}
 \label{GqsplitA} \|G(t) \|_{L^1} \leq \sum_{ 2^q \leq A} \|G_{q} (t)\|_{L^1} + \sum_{ 2^q \geq  A} \|G_{q} (t)\|_{L^1} 
 \end{equation}
   Let us first consider large time estimates that is to say  for $ t \geq A$.   In this case, the second sum gives
   $$ \sum_{2^q \geq A} \|G_{q}(t)\|_{L^1} \lesssim \sum_{2^q \geq A}  { 2^{q( 1 + \delta)} \over 1 + 2^{2 q}} { 1  \over (1 + 2^q |t|)^N}
     \lesssim { 1 \over |t|^N},$$
     since $\delta \leq 1$.
      For the first sum, we split
     $$   \sum_{ 2^q \leq A} \|G_{q} (t)\|_{L^1} \lesssim  \sum_{ 2^q \leq A} { 2^q  \over (1 + 2^q |t|)^N}
      \lesssim   \sum_{ 2^q  \leq t^{-1}}   { 2^q  \over (1 + 2^q |t|)^N} +  \sum_{ t^{-1} \leq  2^q  \leq A} { 2^q  \over (1 + 2^q |t|)^N}.$$
      The first term above contains only negative $q$ so that
      $$  \sum_{ 2^q  \leq t^{-1}}   { 2^q  \over (1 + 2^q |t|)^N} 
       \lesssim  \sum_{ 2^q  \leq t^{-1}}  2^q \lesssim  { 1 \over t}.$$ 
     For the second term, we write 
   $$\sum_{ t^{-1} \leq  2^q  \leq A} { 2^q  \over (1 + 2^q |t|)^N} 
    \lesssim  { 1 \over t^N} \sum_{ t^{-1} \leq  2^q  \leq 1} 2^{q( 1 - N)} +{ 1 \over t^N} \sum_{1 \leq 2^q \leq A} { 2^{q(1-N)}  }
    \lesssim  { 1 \over t} + { 1 \over t^N},
    $$
    since $N>1$.
   We have thus proven that for $t$ large enough
   $$ \|G(t) \|_{L^1} \lesssim { 1 \over t}.$$
  The estimates for the $L^\infty$ norm follows the same lines using that $N>d+1$.
  
  Let us explain  how we obtain the estimates for $t \leq A$. We use again \eqref{GqsplitA}. 
   For the first sum, we just use that
   $$  \sum_{ 2^q \leq A} \|G_{q}(t) \|_{L^1} \lesssim  \sum_{ 2^q \leq A} { 2^q  \over (1 + 2^q |t|)^N} \lesssim 
    \sum_{ 2^q \leq A} 2^q \lesssim 1.$$
     For the second sum, we write 
   $$ \sum_{ 2^q \geq A} \|G_{q}(t)\|_{L^1} \lesssim 
     \sum_{ 2^q \geq A}  { 2^{q( 1 + \delta)} \over 1 + 2^{2 q}} { 1  \over (1 + 2^q |t|)^N}
      \lesssim   \sum_{ 2^q \geq A} 2^{q( - 1 + \delta)} \lesssim 1.$$
  To get the short time estimates for the $L^\infty$ norm, we only handle in a slightly different way the second sum.
   We write
$$ \begin{aligned}\sum_{ 2^q \geq A} \|G_{q}(t)\|_{L^\infty}
   &\lesssim   \sum_{ 2^q \geq A}  { 2^{q(d+ 1 + \delta)} \over 1 + 2^{2 q}} { 1  \over (1 + 2^q |t|)^N}
    \\&\lesssim   \sum_{ t^{-1} \geq  2^q \geq A}   2^{q(d - 1+ \delta)} 
     + { 1 \over t^N} \sum_{  2^q \geq t^{-1}}   2^{q(d - 1 + \delta -N)}   \lesssim  { 1 \over t^{d-1 + \delta}} .
\end{aligned}$$

We can also estimate derivatives of $G$ using the Bernstein inequality~\eqref{Bernstein1}.
 Note that we will use only large time estimates (that is to say for $t\geq 1$). We write
$$ \begin{aligned}
  \|\nabla G(t) \|_{L^1} &\leq \sum_{ 2^q \leq A}  2^q\|G_{q} (t)\|_{L^1} + \sum_{ 2^q \geq  A} 2^q \|G_{q}(t) \|_{L^1} 
  \\ 
  &\lesssim  \sum_{{ 2^q} \leq t^{-1} } 2^{2q } +  { 1 \over t^N}  \sum_{ {t^{-1}} \leq  2^q \leq A}  2^{q (2 -N)} +  { 1 \over t^N }\sum_{2^q \geq A}  { 2^{q( 2 + \delta)} \over 2^{q (2 +N)} }   \lesssim { 1 \over t^2} + { 1 \over t^N}.
 \end{aligned}$$
The estimate for the $L^\infty$ norm of the derivatives follows the same lines, using $N>d+2$.
  This concludes the proof of Theorem \ref{volterrasol}.
 
%
%

\section{The bootstrap argument}\label{sec-nonlin}

Equipped with the linear estimates (in the form of Corollary~\ref{corodecayL}), we are now in position to introduce the continuation argument that we shall use  to 
establish Theorem \ref{theomain}.  
As usual, the characteristics $(X_{s,t}(x,v),V_{s,t}(x,v))$ associated to the transport equation with the vector field $(v, E(t,x))$ are defined as the solution to the ODE system:
\begin{equation}
\label{charac}
\left \{ 
\begin{aligned}
&\frac{d}{ds}X_{s,t}(x,v) = V_{s,t}(x,v), \qquad &X_{t,t}(x,v) =x,\\
&\frac{d}{ds} V_{s,t}(x,v) = E(s,X_{s,t}(x,v)), \qquad &V_{t,t}(x,v) =v.
\end{aligned}
\right.
\end{equation}
By the method of characteristics, the solution to the Vlasov-Poisson system~\eqref{VP} must satisfy
\begin{equation}
\label{charf}
f(t,x,v)= f_0(X_{0,t}(x,v) , V_{0,t}(x,v))  - \int_0^t E(s,X_{s,t}(x,v)) \cdot \na_v \mu(V_{s,t}(x,v))  \,  ds.
\end{equation}
Consequently, $\rho(t,x)= \int_{\R^d} f(t,x,v)\, dv$ solves the equation 
\begin{equation}\label{iter-rho}
\rho(t,x) -  \int_0^t \int_{\R^d}  [\na_x (1 - \Delta_x)^{-1} \rho](s,x -(t-s)v) \cdot \na_v \mu(v) \, dv ds = S(t,x),
\end{equation}
with
$$\begin{aligned}
S(t,x) &= \int_{\R^d} f_0(X_{0,t}(x,v) , V_{0,t}(x,v)) \, dv +  \int_0^t \int_{\R^d}  E(s,x-(t-s)v) \cdot \na_v \mu(v)   \, dv  ds 
\\&\quad - \int_0^t \int_{\R^d}  E(s,X_{s,t}(x,v)) \cdot \na_v \mu(V_{s,t}(x,v))  \, dv  ds.
\end{aligned}$$

To study \eqref{iter-rho}, let us introduce the following weighted in time norm:
\begin{equation}\label{def-iterN}
\begin{aligned}
\mathcal{N}(t) &= 
\sup_{[0,t]}   { 1 \over \log ( 2+ s) }\Big( \|\rho(s)\|_{L^1} + \langle s\rangle^d \|\rho(s) \|_{L^\infty} +\langle s\rangle \| \nabla \rho(s)\|_{L^1} 
+ \langle s\rangle^{d+1} \|\nabla \rho(s)\|_{L^\infty}\Big).
\end{aligned}
\end{equation}
First, we recall the following local well-posedness result for~\eqref{VP} whose proof is standard and therefore omitted.  In what follows, we shall say that  $f\in W^{\sigma,p}_{k}$  if $\langle v \rangle^k f \in W^{\sigma,p}_{x,v}$.
\begin{prop}[Local well-posedness] 
\label{prop-LWP} 
Let  $f_{0} \in W^{1,1} \cap W^{\sigma,p}_{k/p'}$ with $k>d$ and $p(\sigma -1)> 2d$. Then, there exists $T_0>0$ and a unique classical solution $f(t) \in \mathscr{C}([0,T_0], W^{1,1} \cap W^{\sigma,p}_{k/p'})$.
Denote by  $T^\star>0$  the maximal existence time;
if $T^\star<+\infty$, then
\begin{equation}
\label{eq-BU-crit}
 \| \rho \|_{L^1(0,T^\star, W^{1,\infty})} =+\infty.
\end{equation}

%

\end{prop}


We thus apply Proposition~\ref{prop-LWP}, 
 to obtain a unique local solution $f(t) \in \mathscr{C}([0,T_0], W^{1,1} \cap W^{\sigma,p}_{k/p'})$ of~\eqref{VP} 
   that can be continued 
  as long as $\|\rho(t) \|_{W^{1, \infty}}$ remains bounded. We note that the continuity in time  of $f$ and the Sobolev embedding entail that the function $t \mapsto \mathcal{N}(t)$ is continuous as well.

Let $\eps \in (0,1]$ to be fixed later. 
We define
\begin{equation}\label{def-iterTime}
T= \sup \Big\{t  \in (0,T^\star),  \, \mathcal{N}(t) \leq \eps\Big\},
\end{equation}
where $T^\star$ is the maximal time of existence.
By continuity of $\mathcal{N}$ and thanks to the blow-up criterion~\eqref{eq-BU-crit}, there is $\eps_0>0$ such that if
$$ \| \langle v \rangle^k f_{0}\|_{W^{1, \infty}} +   \|  f_{0}\|_{W^{1,1}} \leq \eps_{0},$$
then $T^\star >2$ and moreover we can ensure that  the time defined in~\eqref{def-iterTime}  also satisfies $T>2$.


Now applying the linear theory, namely Corollary \ref{corodecayL},  for the equation \eqref{iter-rho},  we have 
\begin{equation}\label{bd-NNN}
\mathcal{N}(t) 
\leq M'\left(  \|S\|_{Y^0_{t}} +  \|S\|_{Y_{t}^1}\right),
\end{equation}
for $t\le T$, where the norms $Y_t^0, Y_t^1$ are defined as in \eqref{def-Y0Y1}. 

The main task is therefore  to estimate $\|S\|_{Y^0_{t}} +  \|S\|_{Y_{t}^1}$.
We are going to prove the following result.
\begin{thm}
\label{prop-main}
There exist $\eps_1 \in (0,1]$,  $C_0>0$ such that for all $\eps \in (0,\eps_1)$, for all $t \leq T$ (where $T$ is defined as in \eqref{def-iterTime}), 
\begin{equation}
\|S\|_{Y_{t}^0} +  \|S\|_{Y_{t}^1} \leq C_0 \eps_{0}+ C_0 \eps^2.
\end{equation}
\end{thm}

With Theorem~\ref{prop-main} at hand, we can conclude the proof of Theorem~\ref{theomain}.

\begin{proof}[Proof of Theorem~\ref{theomain}]
We   choose $\eps_0$ and $\eps$ small enough so that $\eps_0 \leq \eps \leq \eps_1$ and
\begin{equation}
\label{eq-choiceeps}
 C_0 \eps_0  + C_0 \eps^2 \leq \frac{1}{2M'} \eps,
\end{equation}
where $M'$ is the constant appearing in \eqref{bd-NNN}.

Assume that $T<T^\star$.
We deduce by Theorem~\ref{prop-main}, \eqref{bd-NNN} and~\eqref{eq-choiceeps}
that  
$$\mathcal{N}(T) \leq \frac{1}{2} \eps.$$ 
This is a contradiction with the definition of $T$ as by continuity of $\mathcal{N}$, we would then find $T_1 \in (T,T^\star)$ such that 
$$\mathcal{N}(T_1) \leq  \eps.$$ 
Once we know that $T= T^\star$,  we also get that $T^\star=+\infty$: indeed from  the blow-up criterion of Proposition~\ref{prop-LWP}, 
we cannot have $T^\star<+\infty$ since then $\mathcal{N}(T^\star) \leq \eps$.
 
\end{proof}

\bigskip

The rest of this paper is devoted to the proof of Theorem~\ref{prop-main}. We shall therefore work on the interval $[0,T]$, on which we have $\mathcal{N}(t) \leq \eps$, a property which we will refer to as the bootstrap assumption.
All subsequent estimates will be independent of $T$.

\section{Decay estimates for characteristics}\label{sec-rechar}
In this section, we study the characteristics  $(X_{s,t}(x,v),V_{s,t}(x,v))$ defined as the solutions to the ODEs \eqref{charac}. Integrating \eqref{charac}, we have for all $0\leq s \leq t \leq T$,
\begin{equation}
\label{int-XX} 
\begin{aligned}
X_{s,t}(x,v)  &= x - v(t-s) +  \int_s^t (\tau-s) E \Big(\tau, X_{\tau,t}(x,v)\Big)  \; d\tau
\\
V_{s,t}(x,v)  &= v  -  \int_s^t E \Big(\tau, X_{\tau,t}(x,v)\Big)  \; d\tau.
\end{aligned}
\end{equation}
We have the following pointwise bounds on characteristics. 

\begin{prop}\label{prop-char} There is $\eps_1 \in (0,1)$ such that the following holds for all $\eps\le \eps_1$. For all $0\leq s,t \leq T$ and for $x,v\in \R^d$, the map $x \mapsto X_{s,t}(x,v)$ and $v\mapsto V_{s,t}(x,v)$ are $\mathscr{C}^1$ diffeomorphisms, and we can write\footnote{We have chosen to write the remainder as functions of $(x-vt,v)$, instead of $(x,v)$, in view of the expected large time behavior, which is that of  free transport.} 
\begin{equation}\label{exp-XV} X_{s,t}(x,v)  = x - (t-s)v + \X_{s,t}(x-vt,v), \qquad V_{s,t}(x,v) = v + \V_{s,t}(x-vt,v),\end{equation}
where $\X_{s,t}(x,v), \V_{s,t}(x,v)$ satisfy the following uniform estimates
\begin{equation}
\label{eq-decXV}
\begin{aligned}
 & \sup_{0\leq s\leq t \leq T}  {    1 + s^{ d - 1} \over \log(2 + s)  } ( \|\X_{s,t}\|_{L^\infty_{x,v}} + \|\nabla_{x}\X_{s,t}  \|_{L^\infty_{x,v}} ) +  \sup_{0\leq s\leq t \leq T}  {    1 + s^{ d - 2} \over \log(2 + s)  } \|\nabla_{v}\X_{s,t} \|_{L^\infty_{x,v}} 
 \lesssim
 \eps,\\
 &  
  \sup_{0\leq s \leq t \leq T}  { 1 + s^{d } \over \log ( 2 + s)  }  ( \|\V_{s,t}\|_{L^\infty_{x,v}}  +  \| \nabla_{x} \V_{s,t}\|_{L^\infty_{x,v}} ) +\sup_{0\leq s\leq t \leq T}  {    1 + s^{ d - 1} \over \log(2 + s)  } \| \nabla_{v} \V_{s,t}\|_{L^\infty_{x,v}} \lesssim 
\eps.  
 \end{aligned}
\end{equation}
\end{prop}
\begin{proof}
Let $0\leq s\leq t \leq T$.  By definition of $\X_{s,t}$, we have from~\eqref{int-XX} 
\begin{equation}
\label{int-YY} 
\X_{s,t}(x,v)  = \int_s^t (\tau-s) E \Big(\tau, x + \tau v + \X_{\tau,t}(x,v)\Big)  \; d\tau. \end{equation}
Differentiating this identity with respect to $x$ and taking the sup norm in $x$ and $v$, we get 
\begin{equation}
\label{nablaYst1}
\begin{aligned}
 \|\nabla_x \X_{s,t} \|_{L^{\infty}_{x,v}}  
 &\le  \int_s^t (\tau-s)  \|  \nabla_x E(\tau, x + \tau v + \X_{\tau,t}(x,v))\|_{L^\infty_{x,v}}  ( 1 + \|\nabla_x \X_{\tau,t} \|_{L_{x,v}^{\infty}} )\; d\tau
 \\
 &\le  ( 1 + \sup_{0\le \tau \le t} \| \nabla_x \X_{\tau,t} \|_{L_{x,v}^{\infty}})
 \int_s^t  \tau\|  \nabla_x E(\tau)\|_{L^{\infty}}   \; d\tau 
. \end{aligned}
\end{equation}
Next, since  $\nabla_x E = - \nabla_x (1-\Delta_x)^{-1} (\nabla_x\rho)$, we obtain from  \eqref{riesz1}  and the bootstrap assumption
 that
\begin{equation}
\label{estnablaE1} \| E(\tau)\|_{L^\infty}+ \|\nabla_x E(\tau ) \|_{L^\infty} \lesssim \|\nabla_x \rho(\tau) \|_{L^\infty} \lesssim {   \log( 2 + \tau)  \over  1 + \tau^{1+ d } } \eps,
\end{equation}
and thus $  \tau  \|\nabla_x E(\tau ) \|_{L^\infty} $ is integrable in time. 
 This yields
$$  \|\nabla_x \X_{s,t} \|_{L^{\infty}_{x,v}}  \lesssim \eps( 1 +  \|\nabla_x \X_{s,t} \|_{L^{\infty}_{x,v}}).$$
Thus,  for $\eps$ sufficiently small, we have $\|\nabla_x \X_{s,t} \|_{L^{\infty}_{x,v}}  \lesssim \eps$, for all $0\leq s\leq t \leq T$. 
Therefore, taking $\eps$ small enough, we deduce that for all $v\in \R^d$, the map $x \mapsto X_{s,t}(x,v)$  is a $\mathscr{C}^1$ diffeomorphism. As a consequence, for any integrable function $H = H(x)$, we have 
\begin{equation}\label{change-XX}\sup_{0\leq s,t \leq T} \| H(X_{s,t})\|_{L^\infty_vL^1_x} \lesssim \| H\|_{L^1_x}.\end{equation}

We are ready to derive the uniform estimates in \eqref{eq-decXV}. Using the pointwise decay in \eqref{estnablaE1}, it follows directly from~\eqref{int-YY} and \eqref{nablaYst1} that 
 $$ \| \X_{s,t}\|_{L^\infty_{x,v}} +  \| \nabla_{x} \X_{s,t}\|_{L^\infty_{x,v}}  \lesssim    {  \log(2+ s)   \over  1 + s^{d - 1 }} \eps.$$
In addition, we also obtain from \eqref{int-YY} that
$$\begin{aligned}
    \| \nabla_{v} \X_{s,t}\|_{L^\infty_{x,v}}
&  \lesssim  \int_{s}^t (\tau- s) \| \nabla_{x} E (\tau) \|_{L^\infty} (  \tau + \|\nabla_{v }\X_{\tau, t }\|_{L^\infty})\, d\tau
   \\ &\lesssim  \eps \int_{s}^t {   \log(2+ \tau)   \over  1 + \tau^{d    - 1}} \, d\tau +\eps  \int_{s}^t    { \log(2+ \tau)   \over  1 + \tau^{ d }} 
    \|\nabla_{v }\X_{\tau, t }\|_{L^\infty}\, d\tau.
\end{aligned}$$
    In dimension $d \geq 3$, since  ${  \log(2 + \tau)  \over  1 + \tau^{d   - 1}}$ is integrable, we get from the same argument
    as before that 
   \begin{equation}\label{DdvX}   \| \nabla_{v} \X_{s,t}\|_{L^\infty_{x,v}} \lesssim \eps  {   \log(2+s)  \over  1 + s^{ d  - 2}}\end{equation}
    which is particular bounded, giving the claimed estimates on $\X_{s,t}(x,v)$. Similarly, by construction, we have 
$$     \V_{s,t}(x,v)  = - \int_s^t E \Big(\tau, x + \tau v + \X_{\tau,t}(x,v)\Big)  \; d\tau,$$
which first directly yields
$$
  \| \V_{s,t}\|_{L^\infty_{x,v}}
      \lesssim  \eps \int_{s}^t  {   \log(2 + \tau)  \over  1  + \tau^{d + 1 } } \;d\tau \lesssim   \eps { \log(2 + s)  \over 1 + s^{ d  }}.
$$
Moreover, using the estimates already proved for $\X_{s,t}$, we have 
     $$
\begin{aligned}     
      \| \nabla_{x} \V_{s,t}\|_{L^\infty_{x,v}}
      &\lesssim  \eps \int_{s}^t  {   \log(2 + \tau)  \over  1  + \tau^{d + 1 } } ( 1 + \| \nabla_{x} \X_{\tau,t} \|_{L^\infty_{x,v}} )\, d\tau
     \lesssim   \eps { \log(2 + s)  \over 1 + s^{ d  }}.
\\
 \| \nabla_{v} \V_{s,t}\|_{L^\infty_{x,v}}
      &\lesssim  \eps \int_{s}^t  { \tau  \log(2 + \tau) \over  1 + \tau^{d + 1 } } ( 1 + \| \nabla_{v} \X_{\tau,t} \|_{L^\infty_{x,v}} )\, d\tau
     \lesssim   \eps {  \log(2+s)   \over 1 + s^{d  - 1}} .
\end{aligned}    $$
 Imposing again $\eps$ small enough, we deduce that for all $x\in \R^d$, the map $v \mapsto V_{s,t}(x,v)$  is a $\mathscr{C}^1$ diffeomorphism. The proposition follows. \end{proof}

  \begin{rem}
In dimension $2$, derivatives in $v$ of characteristics
  have a slow growth in time (see in particular the estimates leading to \eqref{DdvX}), which prevents from performing the same nonlinear stability analysis. 
   \end{rem}

\section{Straightening characteristics}
\label{sec-straightXV}

In order to boil down to the case of free transport, we shall rely on a change of variables in velocity, that is close to the identity on the interval $[0,T]$.
This is the content of the following proposition. 

\begin{prop}\label{prop-straight} There is $\eps_1 \in (0,1)$ such that the following holds for all $\eps\le \eps_1$. For all $0\leq s,t \leq T$, there exists a $\mathscr{C}^1$ map  $(x,v) \mapsto \Psi_{s,t}(x,v)$ such that
\begin{equation}\label{inverse-V}
X_{s,t}(x,\Psi_{s,t}(x,v))= x  - (t-s) v,
\end{equation}
for all $x,v \in \R^d$. 
In addition,
for every $x\in \R^d $,  $v \mapsto \Psi_{s,t}(x,v)$ is a diffeomorphism, and there hold the following uniform estimates
\begin{equation}\label{eq-Psi3rd}
\begin{aligned}
\sup_{0\leq s,t \leq T}  {1+s^{d } \over \log(2 + s) } \left(  \| \Psi_{s,t}(x,v)-v\|_{L^\infty_{x,v}}
 +    \|\nabla_{x} \Psi_{s,t}(x,v)\|_{L^\infty_{x,v}} \right) & \lesssim \eps, \\
 \sup_{0\leq s,t \leq T}  {1+s^{d -1}\over \log(2+s)} \|\nabla_{v}( \Psi_{s,t}(x,v)-v)\|_{L^\infty_{x,v}}  &\lesssim \eps.
\end{aligned}\end{equation}
\end{prop}
\begin{proof} From \eqref{int-XX}, we write 
$$ X_{s,t}(x,v) = x - (t-s) \Big( v + \Phi_{s,t}(x,v)\Big)$$ 
where 
$$ \Phi_{s,t}(x,v) =  - \frac{1}{t-s} \int_s^t (\tau-s) E \Big(\tau, x -(t-\tau)v + \X_{\tau,t}(x-vt,v)\Big)  \; d\tau .$$
Using \eqref{estnablaE1}, we obtain 
$$  \| \Phi_{s,t}\|_{L^\infty_{x,v}} \lesssim \int_s^t \|E(\tau) \|_{L^\infty_{x,v}}\; d\tau\lesssim 
\eps \int_{s}^t {  \log(2+ \tau) \over 1 + \tau^{ d+ 1 }} \; d\tau
 \lesssim   \eps  { \log(2 + s) \over 1 + s^{ d }}.$$
In addition, using the bounds on $\X_{s,t}(x,v)$ obtained in the previous section, we get 
$$
\begin{aligned} \| \nabla_x\Phi_{s,t}\|_{L^\infty_{x,v}} 
&\lesssim   \frac{1}{t-s} \int_s^t (\tau-s) \|\nabla_xE(\tau)\|_{L^\infty_x}  \Big(1 + \| \nabla_x\X_{\tau,t}\|_{L^\infty_{x,v}}\Big)  \; d\tau \\
&\lesssim  \int_s^t \|\nabla_xE(\tau)\|_{L^\infty_x} \; d\tau \lesssim \eps  { \log(2 + s) \over 1 + s^{ d }},
\end{aligned}$$
using \eqref{estnablaE1}. Finally, we compute 
$$
\begin{aligned} \| \nabla_v\Phi_{s,t}\|_{L^\infty_{x,v}} 
&\lesssim   \frac{1}{t-s} \int_s^t (\tau-s) \|\nabla_xE(\tau)\|_{L^\infty_x}  \Big(t-\tau + \| (\nabla_v - t\nabla_x)\X_{\tau,t}\|_{L^\infty_{x,v}}\Big)  \; d\tau
\\
&\lesssim \int_s^t (1+\tau) \|\nabla_xE(\tau)\|_{L^\infty_x} \;d\tau + \int_s^t\|\nabla_xE(\tau)\|_{L^\infty_x} \| (\nabla_v - \tau \nabla_x)\X_{\tau,t}\|_{L^\infty_{x,v}}  \; d\tau
\end{aligned}$$
in which we have written $t\nabla_x\X_{\tau,t} = (t-\tau)\nabla_x\X_{\tau,t} + \tau \nabla_x\X_{\tau,t}$ and used the boundedness of $\nabla_x \X_{\tau,t}(x,v)$. Therefore, using again \eqref{estnablaE1} and \eqref{eq-decXV}, we have 
$$
\begin{aligned} \| \nabla_v\Phi_{s,t}\|_{L^\infty_{x,v}} 
&\lesssim \eps \int_{s}^t {  \log(2+ \tau) \over 1 + \tau^{ d}} \;d\tau
 \lesssim   \eps  { \log(2 + s) \over 1 + s^{ d -1}} .
\end{aligned}$$

We deduce that the map 
$$
(x,v) \mapsto (x,v + \Phi_{s,t}(x,v))
$$ 
is a $\mathscr{C}^1$ diffeomorphism from $\R^{2d}$ to $\R^{2d}$ with  Jacobian determinant close to one. As a consequence, there exists a $\mathscr{C}^1$ diffeomorphism $\Psi_{s,t}(x,v)$ such that 
$$X_{s,t}(x,\Psi_{s,t}(x,v))= x  - (t-s) v,
$$
for all $x,v$. By construction, note that we have 
\begin{equation}
\label{id-Psi123}
 \Psi_{s,t}(x,v) = v - \Phi_{s,t}(x,\Psi_{s,t}(x,v)).
  \end{equation}
  The estimates~\eqref{eq-Psi3rd} for  $\Psi_{s,t}(x,v)$ thus follow from the estimates already proved for  $\Phi_{s,t}(x,v)$. This ends the proof of the proposition.
\end{proof}

\section{Contribution of the initial data}\label{sec-bddata}

In this section, we estimate the contribution of 
\begin{equation}
\label{def-Ire}
\mathcal{I}(t,x) = \int_{\R^d} f_0(X_{0,t}(x,v) , V_{0,t}(x,v))  \; dv.
\end{equation}
Namely, we prove the following proposition.

\begin{prop}\label{keyprop-initial}
For $0\le t\le T$, there holds 
\begin{equation}
 \|\mathcal{I}(t) \|_{L^1} + \langle t\rangle^d \|\mathcal{I}(t) \|_{L^\infty}  +  \langle t\rangle\|\nabla_x\mathcal{I}(t) \|_{L^1} + \langle t\rangle^{d+1} \|\nabla_x\mathcal{I}(t) \|_{L^\infty}  \le C_0 \eps_0,
\end{equation}
for some positive constant $C_0$. 
\end{prop}

\begin{proof}
 Since $$(x,v) \mapsto (X_{0,t}(x,v), V_{0,t}(x,v))$$ is a measure preserving diffeomorphism, we have that
 $$ \| \mathcal{I}(t)\|_{L^1} \leq \| f_{0}\|_{L^1_{x,v}} \leq \eps_{0}.$$ 
 Next, to prove the $L^\infty$ estimate,  
 we aim at getting as close as possible to the dynamics of free transport, which corresponds to $f_0(x-tv,v)$. 
To proceed, we first straighten the characteristics following Proposition~\ref{prop-straight}, yielding
$$
\mathcal{I}(t,x) = \int_{\R^d} f_0(x-tv , V_{0,t}(x,\Psi_{0,t}(x,v))) \det (\na_v \Psi_{0,t}(x,v))  \; dv,
$$
from which we deduce by  using the change of variables $w=x-tv$
\begin{equation}
\label{eq-I-0}
\begin{aligned}
\mathcal{I}(t,x) &=  \int_{\R^d} f_0\left(w, V_{0,t}\left(x,\Psi_{0,t}\left(x,\frac{x-w}{t}\right)\right)\right) \det \left(\na_v \Psi_{0,t}\right)\left(x,\frac{x-w}{t}\right)  \; \frac{dw}{t^d} .
\end{aligned}
\end{equation}
 Then, by using \eqref{eq-Psi3rd} that yields
$$
\|  \det \left(\na_v \Psi_{0,t}\right)\|_{L^\infty_{x,v}} \lesssim 1,
$$
we get that
 $$ \|\mathcal{I}(t) \|_{L^\infty} \lesssim
  { 1 \over t^d} \int_{\mathbb{R}^d}   \|f_{0}(w, \cdot)\|_{L^\infty_{v}} \, dw \lesssim { 1 \over t^d} \eps_{0}.$$
  
  Next, to compute the derivative of $\mathcal{I}$, we shall introduce $(\X_{0,t}, \V_{0,t})$ as defined in \eqref{exp-XV} on the characteristics and use the change of variables $w = x-vt$ to get 
  \begin{equation}
\label{def-Ire2}
\begin{aligned}
\mathcal{I}(t,x) &= \int_{\R^d} f_0\left(x-vt + \X_{0,t}(x-tv,v), v +\V_{0,t}(x-tv,v)\right)  \; dv \\
&= \int_{\R^d} f_0\left(w + \X_{0,t}\left(w,\frac{x-w}{t}\right), \frac{x-w}{t} + \V_{0,t}\left(w,\frac{x-w}{t}\right)\right)  \; \frac{dw}{t^d}.
\end{aligned}
\end{equation}
 Therefore, we obtain that
 $$\begin{aligned} 
  \nabla_{x}\mathcal{I}(t,x)
 = \int_{\mathbb{R}^d} \Bigg[(\nabla_{x} f_{0})& \left(w + \X_{0,t}\left(w,\frac{x-w}{t}\right), \frac{x-w}{t} + \V_{0,t}\left(w,\frac{x-w}{t}\right)\right) 
  \cdot (\nabla_{v} \X_{0,t}) \left(w,\frac{x-w}{t}\right) 
   \\  &\quad+ ~(\nabla_{v} f_{0}) \left(w + \X_{0,t}\left(w,\frac{x-w}{t}\right), \frac{x-w}{t} + \V_{0,t}\left(w,\frac{x-w}{t}\right)\right)  \\
   &\qquad\qquad\qquad\qquad\qquad\qquad\qquad \cdot \left( e + (\nabla_{v} \V_{0,t})
    \left(w,\frac{x-w}{t}\right) \right) \Bigg] \; \frac{dw}{t^{d+1}}, 
    \end{aligned}$$
    where $e$ is the vector $(1,\cdots, 1)^t$.
    
 Consequently, from Proposition \ref{prop-char}, we obtain that
 \begin{equation}
\label{eq-naI}
 |\nabla_{x}\mathcal{I}(t,x)|
  \lesssim \int_{\mathbb{R}^d} \left| \nabla_{x,v} f_{0}\right|  \left(w + \X_{0,t}\left(w,\frac{x-w}{t}\right), \frac{x-w}{t} + \V_{0,t}\left(w,\frac{x-w}{t}\right)\right)   \; \frac{dw}{t^{d+1}}.
\end{equation}
Going back to the original coordinates, we get $$
  |\nabla_{x}\mathcal{I}(t,x)|  
 \lesssim  { 1 \over t }\int_{\mathbb{R}^d} |\nabla_{x, v} f_{0}| (X_{0,t}(x,v) , V_{0,t}(x,v) )  \; dv.
   $$
  By integrating in $x$, this yields
  $$ \|\nabla_{x}  \mathcal{I}(t) \|_{L^1} \lesssim { 1 \over t} \|\nabla_{x,v} f_{0}\|_{L^1}
   \lesssim { \eps_{0} \over t}.$$
   For the $L^\infty$ norm, 
   we use again the straightening change of variables    $ v= \Psi_{0,t}(x, \widetilde v)$ and
    $ w=x- t \widetilde v$ to obtain
 $$   |\nabla_{x}\mathcal{I}(t,x)|  
 \lesssim  { 1 \over t^{ d+1} } \int_{\mathbb{R}^d}  \| \nabla_{x, v} f_{0}(w, \cdot) \|_{L^\infty_v} \, dw
  \lesssim  { \eps_{0} \over t^{ d+1} }.$$

This yields the proposition for the case $t\ge 1$. For $0\le t\le 1$, the estimates are straightforward, using directly the bounds on the characteristics from Proposition \ref{prop-char}.   
\end{proof}

\section{Contribution of the reaction term}\label{sec-bdreaction}

We next turn to the reaction term 
\begin{equation}\label{def-reaction}
\mathcal{R}(t,x) =  \mathcal{R}_{\text{L}}(t,x) - \mathcal{R}_{\text{NL}}(t,x)
\end{equation}
where 
$$ \begin{aligned}
\mathcal{R}_{\text{NL}}(t,x)  &= \int_0^t \int_{\R^d} E(s,X_{s,t}(x,v)) \cdot \na_v \mu(V_{s,t}(x,v))  \, dv  ds,
\\
\mathcal{R}_{\text{L}}(t,x) 
 & = \int_0^t \int_{\R^d} E(s,x-(t-s)v) \cdot \na_v \mu(v) \, dv ds.
 \end{aligned}
$$
In order to estimate $\mathcal{R}$,  we shall first establish the following general fact.
 Let us set for any  given $F$ and $\nu$: 
 $$ \mathcal{T}[F, \nu](t,x)=  -\int_0^t \int_{\R^d} F(s,X_{s,t}(x,v)) \cdot \na_v \nu(V_{s,t}(x,v))  \, dv  ds
 +  \int_0^t \int_{\R^d} F(s,x-(t-s)v) \cdot \na_v \nu(v) \, dv ds.$$
 Note that in particular, we have  $\mathcal{R}(t,x)= \mathcal{T}[E, \mu](t,x).$  We have
 \begin{lem}
 \label{lem-decayR}
 Assume that $\nu \in W^{3,\infty}_v$ with
\begin{equation}
\label{decay-nu}
 |\pa^\alpha_v \na \nu| \lesssim \frac{1}{\langle v \rangle^N}, \quad |\alpha|=1,2,
\end{equation}
 for some $N>d$.
  We have the decomposition
  \begin{equation} 
  \label{TFdec} \mathcal{T}[F, \nu](t,x)=   \mathcal{T}_{1}[F, \nu](t,x)+  \mathcal{T}_{2}[F, \nu](t,x),
  \end{equation}
   where for $j=1, \, 2$, we have
$$   \mathcal{T}_{j} [F, \nu](t,x)=  \int_0^t \int_{\R^d} F(s,x-(t-s)v) \cdot H_{0,j}[\nu] (s,t,x,v),$$
 in which 
\begin{equation}\label{def-H0j}
\begin{aligned}
H_{0,1}[\nu](s,t,x,v) &=  \na_v \nu(v)  - \na_v \nu(V_{s,t}(x,\Psi_{s,t}(x,v))) 
 \\
H_{0,2}[\nu](s,t,x,v)  &= \na_v \nu(V_{s,t}(x,\Psi_{s,t}(x,v)))\Big[ 1- \det (\na_v \Psi_{s,t}(x,v))  \Big].
\end{aligned}
\end{equation}
Moreover, for  $0 \leq s \leq t \leq T$, the kernels enjoy the uniform  estimate
\begin{equation}
\label{kernel0} |H_{0,j}[\nu] (s,t,x,v)| \lesssim { \eps \over { (1 + |v|)^N}} { \log (2 + s) \over 1+ s^{ d -1}}
\end{equation}
and there exists $C$ such that for every $F$ and for every $t \in (0, T]$, 
\begin{align}
\label{TcontL1}
&  \| \mathcal{T}_{j}[F, \nu](t)\|_{L^1}  \leq C \eps \sup_{[0, t]}{\|F(s) \|_{L^1} \over \log(2+s)} , \\
&  \label{TcontLinfty} \| \mathcal{T}_{j}[F, \nu](t)\|_{L^\infty} \leq { C \eps \over t^d} \sup_{[0, t]}\left( {\|F(s) \|_{L^1} \over \log(2+s)}
  +   { 1+s^{d}\|F(s) \|_{L^\infty} \over \log(2+s)} \right).
  \end{align} 
 \end{lem}

\begin{proof}
Using  the change of variables provided by   Proposition~\ref{prop-straight}, we obtain
$$
\begin{aligned}
\mathcal{T}[F, \nu](t,x ) &= - \int_0^t \int_{\R^d} F(s,x-(t-s)v) \cdot \na_v \nu(V_{s,t}(x,\Psi_{s,t}(x,v))) \det (\na_v \Psi_{s,t}(x,v)) \, dv  ds
 \\&\quad +   \int_0^t \int_{\R^d} F(s,x-(t-s)v) \cdot \na_v \nu(v) \, dv ds.
\end{aligned}$$
The decomposition \eqref{TFdec} follows. Moreover, for  $0 \leq s \leq t \leq T$, the kernels enjoy the estimate \eqref{kernel0}.  
This follows from  Propositions \ref{prop-char} and \ref{prop-straight}: by the mean value inequality, \eqref{decay-nu} and~\eqref{eq-decXV}--\eqref{eq-Psi3rd}, we have
\begin{align*}
|H_{0,1}[\nu](s,t,x,v)| &\lesssim \frac{1}{\langle v \rangle^N} | v- V_{s,t}(x,\Psi_{s,t}(x,v))| \\
&\lesssim \frac{1}{\langle v \rangle^N} \left[| v- V_{s,t}(x,v)|  + | V_{s,t}(x,v)- V_{s,t}(x,\Psi_{s,t}(x,v))| \right] \\
&\lesssim \frac{1}{\langle v \rangle^N} \left[| v- V_{s,t}(x,v)|  +  \|\na_v V_{s,t}\|_{L^\infty_{x,v}} |v-\Psi_{s,t}(x,v)| \right] \\
&\lesssim  { \eps \over {\langle v \rangle^N} }{ \log (2 + s) \over 1+ s^{ d}}.\end{align*}
We handle $H_{0,2}$ similarly.

Next, using the change of variable $w= x- (t-s)v$ and the bounds on the kernels, we obtain 
$$ |\mathcal{T}_j[F, \nu](t,x)| \lesssim \eps \int_{0}^t \int_{\mathbb{R}^d}
 |F(s, w) | \, \left(1 + \left|{x-w \over t-s}\right| \right)^{-N} { \log (2 + s) \over 1+s^{ d -1}}
  {dw ds \over (t-s)^d}.$$
  This yields
 $$ \|\mathcal{T}_j[F,\nu](t) \|_{L^1} \lesssim \eps \int_{0}^t \|F(s) \|_{L^1} { \log (2 + s) \over 1+s^{ d -1}}\, ds
  \lesssim  \eps \sup_{[0,t]} {\|F(s)\|_{L^1} \over \log (2 + s)}   \int_{0}^t { \log^2(2 + s) \over 1+s^{ d -1}}\, ds,$$
   hence the $L^1$ estimate.
   
   For the $L^\infty$ norm, we use
   $$ \|\mathcal{T}_j[F, \nu](t) \|_{L^\infty} \lesssim   { \eps \over t^d} \int_{0}^{t \over 2} \|F(s) \|_{L^1} { \log (2 + s) \over 1+s^{ d -1}}
   \, ds + \eps \int_{t\over 2}^t  \|F(s) \|_{L^\infty}  { \log (2 + s) \over 1+s^{ d -1}}\, ds,
   $$
   which gives
   $$   \|\mathcal{T}_j[F, \nu](t) \|_{L^\infty}
    \lesssim { \eps \over t^d}\left(  \sup_{[0, t/2]} {\| F \|_{L^1} \over \log(2+s)} + \sup_{[t/2, t]}{ 1+s^d \over \log(2+s)} \|F(s) \|_{L^\infty}  \right) \int_{0}^t { \log^2 (2 + s) \over 1+s^{ d -1}}\, ds.$$
 This concludes the proof of the lemma.
  \end{proof}
  
   As an application of Lemma~\ref{lem-decayR}, we can now derive appropriate estimates for $\mathcal{R}$.
 \begin{prop}
 \label{propR}
  For all $0 \leq t \leq T$,
  $$ \|\mathcal{R}(t) \|_{L^1} + (1+t^d) \|\mathcal{R}(t) \|_{L^\infty} \lesssim \eps^2.$$
 \end{prop} 
 \begin{proof}
  Since we have $\mathcal{R}(t,x)= \mathcal{T}[E, \mu](t,x)$, the estimate for $1\leq t \leq T$ follows from (H1), \eqref{TcontL1}, 
   \eqref{TcontLinfty}, the fact that
   $$ \|E(t) \|_{L^p} \lesssim \|\rho(t)\|_{L^p}, $$
for $p\in [1,\infty]$, and the bootstrap assumption. The estimates are straightforward for $0\le t\le 1$, bounding directly \eqref{def-reaction} using the bounds on the characteristics from Proposition \ref{prop-char}. 
 \end{proof}
 We shall now estimate derivatives of $\mathcal{R}$.
 \begin{prop}
 \label{propnablaR}
  There exists $C>0$ such that for every $0 \leq t \leq T$
  $$ (1+t) \|\nabla \mathcal{R}(t) \|_{L^1} + (1+ t^{d+1}) \|\nabla\mathcal{R}(t) \|_{L^\infty} \leq C \eps^2.$$
 \end{prop} 
 \begin{proof} The estimates for $0\le t\le 1$ follow directly from \eqref{def-reaction} and the bounds on the characteristics from Proposition \ref{prop-char}. We therefore focus on the case  $t\ge 1$. We shall first  express $\nabla \mathcal{R}_{\text{NL}}$ in an appropriate way. Using \eqref{exp-XV}, we have
  \begin{align*} 
\mathcal{R}_{\text{NL}}(t,x )  & =\int_0^t \int_{\R^d} E(s, \X_{s,t}( x- tv, v)+ x-tv + sv)
 \cdot \nabla_v\mu (\V_{s,t}(x-tv, v) + v) \, dv  ds
\\
& =  \int_0^t \int_{\R^d} E\left(s,  \X_{s,t}(w,  { x- w \over t}) + w + {s \over t} (x-w)\right)   \cdot \nabla_v\mu \left(\V_{s,t}(w, {x- w \over t}) + { x- w \over t}\right) \, {dw  ds \over t^d}.
 \end{align*}
By taking the gradient in $x$, we next obtain that
$$
\partial_{j} \mathcal{R}_{\text{NL}}(t,x ) = \mathcal{R}_{j}^1 + \mathcal{R}_{j}^2 +\mathcal{R}_{j}^3 + \mathcal{R}_{j}^4,
$$
where
\begin{align*}
    \mathcal{R}_{j}^1    &=    \int_{0}^t \int_{\mathbb{R}^d}
   s \partial_{j} E  \left(s,  \X_{s,t}(w,  { x- w \over t}) +w+ {s \over t} (x-w)\right) \cdot   \nabla_v\mu \left(\V_{s,t}(w, {x- w \over t}) + { x- w \over t}\right) \, {dw  ds \over t^{d+1}}, \\
  \mathcal{R}_{j}^2    &=    \int_{0}^t \int_{\mathbb{R}^d}
 E \cdot \partial_{v_{j}} \nabla_v\mu  \left(\V_{s,t}(w, {x- w \over t}) + { x- w \over t}\right) \, {dw  ds \over t^{d+1}}, \\
\mathcal{R}_{j}^3 & = \int_{0}^t \int_{\mathbb{R}^d}  \left(\partial_{v_{j}} \X_{s,t} (w,  { x- w \over t}) \cdot 
  \nabla_x E \left(s,  \X_{s,t}(w,  { x- w \over t}) +w + {s \over t} (x-w)\right) \right)  \\
 &  \qquad\qquad\qquad\qquad\qquad  \qquad\qquad\qquad\qquad\qquad    \cdot 
   \nabla_v\mu \left(\V_{s,t}(w, {x- w \over t}) + { x- w \over t}\right) \, {dw  ds \over t^{d+1}}, \\
 \mathcal{R}_{j}^4   &  =    \int_{0}^t \int_{\mathbb{R}^d} E\left(s,  \X_{s,t}(w,  { x- w \over t}) +w+ {s \over t} (x-w)\right)
 \\
 &    \qquad\qquad\qquad\qquad\qquad 
  \cdot \left( \partial_{v_{j}} \V_{s,t} (w,  { x- w \over t}) \cdot \nabla_v \right) \nabla_v\mu \left(\V_{s,t}(w, {x- w \over t}) + { x- w \over t}\right)
    {dw  ds \over t^{d+1}}.
\end{align*}
We can now estimate directly  $\mathcal{R}_{j}^3$, $\mathcal{R}_{j}^4$ that are indeed nonlinear terms.
  Going back to the $v$ variable and recalling \eqref{exp-XV}, we have that
  $$\mathcal{R}_{j}^3  = \int_{0}^t \int_{\mathbb{R}^d} \left( \partial_{v_{j}} \X_{s,t} (x-vt,  v) \cdot 
  \nabla_x E \left(s, X_{s,t}(x,v) \right)  \right)  \cdot \nabla_v\mu \left(V_{s,t}(x,v)\right) \, {dv  ds \over t}.$$
From \eqref{eq-decXV}, we then obtain the pointwise estimate
$$
\begin{aligned}
| \mathcal{R}_{j}^3(t,x)| &\lesssim \eps \int_{0}^t   \int_{\mathbb{R}^d} \left| \nabla_x E (s, X_{s,t}(x,v))\right|
 |\na_v \mu(V_{s,t}(x,v)|  { \log(2+ s) \over 1+ s^{d-2}} {dv  ds \over t}
 \\&\lesssim \eps \int_{0}^t  \|\nabla_x E (s)\|_{L^\infty} { \log(2+ s) \over 1+ s^{d-2}}   \int_{\mathbb{R}^d} 
 |\na_v \mu(V_{s,t}(x,v)|  {dv  ds \over t} .
\end{aligned} $$
By Proposition~\ref{prop-char}, for all $x \in \R^d$, $v \mapsto V_{s,t}(x,v)$ is a $\mathscr{C}^1$ diffeomorphism and therefore by~\eqref{eq-decXV} and (H1),
$$
\sup_x \int_{\mathbb{R}^d} 
 |\na_v \mu(V_{s,t}(x,v)|  dv \lesssim  1.
$$
By  \eqref{riesz1} and the bootstrap assumption, we have
$$
 \|\nabla_x E (s)\|_{L^1}   \lesssim  \|\nabla_x \rho (s)\|_{L^1}  \lesssim \eps{ \log(2+ s) \over 1+ s}
$$
and thus the above yields
  $$  \| \mathcal{R}_{j}^3(t)\|_{L^1} \lesssim { \eps \over t} \int_{0}^t   \|\nabla_x E (s)\|_{L^1}  { \log(2+ s) \over 1+ s^{d-2}}  \; ds \lesssim { \eps^2 \over t} \int_{0}^t     { \log^2(2+ s) \over 1+ s^{d-1}}  \; ds
   \lesssim { \eps^2 \over t} .$$
 For the $L^\infty$ norm, we use again (H1) and the change of variable given by 
 Proposition \ref{prop-straight} to obtain that
 $$ | \mathcal{R}_{j}^3(t,x)| \lesssim {\eps \over t} \int_{0}^t   \int_{\mathbb{R}^d} \left|  \nabla_x E (s, w)\right| 
\left(1 + \left|{x-w \over t-s}\right| \right)^{-N} { \log (2 + s) \over 1+ s^{d-2}}
  {dw ds \over (t-s)^d}.$$
Using the bootstrap assumption, we find
  $$
  \begin{aligned}
     | \mathcal{R}_{j}^3(t,x)| 
     &\lesssim { \eps \over t^{d+1}}  \int_{0}^{t \over 2}
   \| \nabla_x E (s)\|_{L^1} \,   { \log (2 + s) \over 1+ s^{d-2}}
    +{\eps \over t} \int_{t \over 2}^t  \| \nabla_x E(s) \|_{L^\infty}    { \log (2 + s) \over  1+ s^{d-2}}
    \\
    &\lesssim { \eps^2 \over t^{d+1}}  \int_{0}^{t \over 2}
 { \log^2 (2 + s) \over 1+ s^{d-1}}
     \lesssim { \eps^2\over t^{d+1}}.
     \end{aligned}$$
     The estimates for $\partial_{j}\mathcal{R}_{j}^4$ are obtained in the same way.
     
     We shall now estimate $ \mathcal{R}_{j}^{1},$ $  \mathcal{R}_{j}^{2}, $
      together with $\partial_{j} \mathcal{R}_{L}$. Recalling \eqref{exp-XV} and going back to the variable $v= (x-w)/t$, we write 
    \begin{align*}
     &   \mathcal{R}_{j}^1     =    \int_{0}^t \int_{\mathbb{R}^d}
   s \partial_{j} E  \left(s,  X_{s,t}(x,  v)\right) \cdot   \nabla_v\mu \left(V_{s,t}(x,v)\right) \, {dw  ds \over t}, \\
&  \mathcal{R}_{j}^2    =    \int_{0}^t \int_{\mathbb{R}^d}
 E \left(s,  X_{s,t}(x,  v)\right) \cdot \partial_{v_{j}} \nabla_v\mu  \left(V_{s,t}(x,v)\right) \, {dw  ds \over t}.
\end{align*}
Next, using  again the straightening change of variable of Proposition \ref{prop-straight}, we obtain
  \begin{align*}
     &   \mathcal{R}_{j}^1     =    \int_{0}^t \int_{\mathbb{R}^d}
   s \partial_{j} E  \left(s,  x-(t-s)v \right) \cdot   \nabla_v\mu \left(V_{s,t}(x,\Psi_{s,t}(x,v)\right)
   \mbox{det}(\nabla_{v} \Psi_{s,t}(x,v)) \, {dv  ds \over t}, \\
&  \mathcal{R}_{j}^2    =    \int_{0}^t \int_{\mathbb{R}^d}
 E \left(s,  x-(t-s)v\right) \cdot \partial_{v_{j}} \nabla_v\mu  \left(V_{s,t}(x,\Psi_{s,t}(x,v)\right)
  \mbox{det}(\nabla_{v} \Psi_{s,t}(x,v)) 
  \, {dv  ds \over t}.
\end{align*}
On the other hand, using also the change of variable $w = x- t v$, we obtain that 
 $$  \mathcal{R}_{\text{L}}(t,x) 
  = \int_0^t \int_{\R^d} E(s, w + {s \over t }( x- w)) \cdot \na_v \mu( { x- w \over t}) \, {dwds\over t^d}$$
  and therefore 
  $$\partial_{j}\mathcal{R}_{\text{L}}(t,x) = \mathcal{R}_{\text{L},j}^1 +   \mathcal{R}_{\text{L},j}^2$$
   where
 \begin{align*}
 \mathcal{R}_{\text{L},j}^1 &=
  \int_0^t \int_{\R^d} (s \partial_{j} E)(s, w + {s \over t }( x- w)) \cdot \na_v \mu( { x- w \over t}) \, {dwds\over t^{d+1}}
  \\
  & = \int_{0}^t \int_{\R^d}  (s\partial_{j}E)(s, x-(t-s)v) \cdot \partial_{v_{j}} \na_v \mu( v) \, {dwds\over t} \\
 \mathcal{R}_{\text{L},j}^2 &=
  \int_0^t \int_{\R^d} E(s, w + {s \over t }( x- w)) \cdot \partial_{v_{j}} \na_v \mu( { x- w \over t}) \, {dwds\over t^{d+1}} \\
  & = \int_{0}^t \int_{\R^d} E(s, x-(t-s)v) \cdot \partial_{v_{j}} \na_v \mu( v) \, {dwds\over t} .
 \end{align*}

 Consequently, we observe that
 $$  - \mathcal{R}_{j}^1 + \mathcal{R}_{\text{L},j}^1 ={ 1 \over t} \mathcal{T}[s \partial_{j} E , \mu](t,x), 
  \quad - \mathcal{R}_{j}^2 + \mathcal{R}_{\text{L},j}^2 ={ 1 \over t} \mathcal{T}[ E , \partial_{j}\mu](t,x),$$
  so that by (H1), we can apply Lemma~\ref{lem-decayR}.
Therefore, by \eqref{TcontL1}, \eqref{TcontLinfty} and the bootstrap assumption, 
   we get that
   $$  \|- \mathcal{R}_{j}^1(t) + \mathcal{R}_{\text{L},j}^1(t) \|_{L^1} \lesssim { \eps^2 \over t}
    , \quad  \|- \mathcal{R}_{j}^2(t) + \mathcal{R}_{\text{L},j}^3(t) \|_{L^\infty} \lesssim { \eps^2 \over t^{d + 1}},$$
     which allows to end the proof.   
     \end{proof}
   

%
%
%
%
   We are finally in position to conclude.
   
\begin{proof}[Proof of Theorem \ref{prop-main}]
This follows from the estimates of Proposition \ref{keyprop-initial}, Proposition \ref{propR}, and Proposition \ref{propnablaR}. 
 \end{proof}
 
 \section{Proof of Corollary \ref{coro1}}
 \label{lastsection}
 We can use again that $f$ is given by \eqref{charf} from the characteristic method. We can also  observe that
 $$\int_0^t E(s,X_{s,t}(x,v)) \cdot \na_v \mu(V_{s,t}(x,v))  \,  ds=  \int_0^t {d \over ds}V_{s,t}(x,v) \cdot \na_v \mu(V_{s,t}(x,v))  \,  ds
  = \mu(v) - \mu(V_{0, t}(x,v)).$$
  By using $(Y_{s,t}, W_{s,t})$ that  are defined in \eqref{exp-XV}, we can thus write that 
  \begin{equation} 
\label{fin1}   f(t, x+tv, v)  = f_{0}(x+ Y_{0,t}(x,v), v+W_{0,t}(x,v)) + \mu( v+ W_{0,t}(x, v)) - \mu(v).
  \end{equation}
  Consequently, the result follows immediately if we prove that there exists $(Y_\infty(x,v), W_{\infty}(x,v))$ such that
  $$ \| Y_{0,t}(x,v) -  Y_{\infty}(x,v)\|_{L^\infty_{x,v}} \lesssim \eps_{0} { \log (2+ t )\over 1 + t^{d-1}}, \quad
       \| W_{0,t}(x,v) -  W_{\infty}(x,v)\|_{L^\infty_{x,v}} \lesssim \eps_{0} { \log (2+ t) \over 1 + t^{d}}.$$
       We give the proof for  $Y_{0,t}$, the one for $W_{0,t}$ being similar.
       
        Let us define $\mathbb{Y}_{t}(s,x,v)= Y_{s,t}(x,v) \mathrm{1}_{0 \leq s \leq t}$.  
From  the integral equation \eqref{int-YY} and  the decay estimates \eqref{eq-thm}, we get that
          $\mathbb{Y}_{t} \in \mathscr{C}_{b}(\R^+ \times \mathbb{R}^{2d})$ and that for every $t_{1} \geq t_{2} \geq 1$, 
        $$ \|\mathbb{Y}_{t_{1}}  -\mathbb{Y}_{t_{2}} \|_{\mathscr{C}_{b}(\R^+ \times \mathbb{R}^{2d})}
         \lesssim   \eps_{0} { \log (2+ t_{2} )\over 1 + t_{2}^{d-1}}  +  \eps_{0} \|\mathbb{Y}_{t_{1}}  - 
          \mathbb{Y}_{t_{2}} \|_{\mathscr{C}_{b}(\R^+ \times \mathbb{R}^{2d})}.$$
          For $\eps_{0}$ sufficiently small, this yields by the Cauchy criterion that 
          $\lim_{t \rightarrow + \infty} \mathbb{Y}_{t}
           := \mathbb{Y}_{\infty}$ exists in $\mathscr{C}_{b}(\R^+ \times \mathbb{R}^{2d})$ and then that
           $$  \|\mathbb{Y}_{t}  - \mathbb{Y}_{\infty} \|_{\mathscr{C}_{b}(\R^+ \times \mathbb{R}^{2d})}
         \lesssim   \eps_{0} { \log (2+ t )\over 1 + t^{d-1}}.$$
          We conclude by setting $Y_{\infty}= \mathbb{Y}_{\infty}(0, \cdot).$
  
~\\
{\em Acknowledgements:}  TN was partially supported by the NSF under grant DMS-1764119 and an AMS Centennial Fellowship, and FR by the ANR ODA and Singflows. Part of this work was done while TN was visiting Princeton University.  

 \appendix 
 \section{}
 
 We recall~\eqref{defLP} for the Littlewood-Paley decomposition in $\mathbb{R}^n$ 
 (for $n=d$ and for $n= d+1$). Let us state the classical Bernstein Lemma.
 
 \begin{lem}
 \label{lemLp}
  For every $p \in [1, + \infty]$ {and any multi-index $\alpha$}, there exist $c>0$,  $C>0$ such that for every $u \in L^p$, we have  Bernstein's inequalities:
 \begin{eqnarray}
 \label{Bernstein1}  c2^{ |\alpha | q}\|u\|_{L^p} \leq \| \partial^\alpha (u_q )\|_{L^p} \leq C 2^{ |\alpha | q}\|  u \|_{L^p}, \quad
  \forall q \in \mathbb{Z}.
 \end{eqnarray}
 \end{lem}
  We refer for example to \cite{Bahouri-Chemin-Danchin} for the proof.
 As an application, we get
  \begin{lem}
 \label{rhoE}
 Let  $P_{1}$ and $P_{2}$  be  homogeneous polynomials of degree  $1$  and $2$.
     For all $p \in [1,+\infty]$, for all $u \in L^p$, for all $\ell \in \mathbb{N}$,
  \begin{eqnarray}
&  \label{riesz1} \|P_{1}(D) ( 1 - \Delta)^{-1} u \|_{L^p} \lesssim \|u \|_{L^p}, \\
 & \label{riesz2}  \|P_{2} (D) ( 1 - \Delta)^{-1} u_\ell \|_{L^p} \lesssim 2^{\ell \delta} \| u \|_{L^p},
  \end{eqnarray}
  where $u_\ell$ is defined as in~\eqref{defLP} and $\delta \in (0,1)$ is arbitrarily small.
  \end{lem}
  \begin{proof}
   By using the homogeneous Littlewood-Paley decomposition and the Bernstein inequalities, we get 
   $$   \|P_{1}(D) ( 1 - \Delta)^{-1}  u \|_{L^p} \lesssim \sum_{q \in \mathbb{Z}}
    { 2 ^{q} \over 1 + 2^{2 q} } \| u_{q}\|_{L^p}
     \lesssim \|u \|_{L^p} (\sum_{q \leq 0} 2^q + \sum_{q \geq 0} 2^{-q}).$$
      For the second estimate, we write
     $$
     \begin{aligned}  \|P_{2}(D) ( 1 - \Delta)^{-1} w \|_{L^p} 
     &\lesssim \sum_{q \in \mathbb{Z}}
    { 2 ^{2q} \over 1 + 2^{2 q} } \|w_{q}\|_{L^p}
      \\&\lesssim   \|w\|_{L^p}  \sum_{q < 0} 2^{2q} + \sup_{q \geq 0} 2^{q \delta} \|w_{q}\|_{L^p}  \sum_{q \geq 0} 2^{-q \delta}
      \end{aligned}$$
     and apply it for $w= u_\ell$,
     which ends the proof. 
     \end{proof}
%

\bibliographystyle{plain}

~\\
{\sc Centre de Math\'ematiques Laurent Schwartz (UMR 7640), Ecole Polytechnique, Institut Polytechnique de Paris, 91128 Palaiseau Cedex, France. 
\\Email: daniel.han-kwan@polytechnique.edu}

~\\
{\sc Department of Mathematics, Penn State University, State College, PA 16803, USA. 
\\Email: nguyen@math.psu.edu}

~\\
{\sc Laboratoire de Math\'ematiques d'Orsay (UMR 8628), Universit\'e Paris-Sud, 91405 Orsay Cedex, France. Email: frederic.rousset@math.u-psud.fr}

\end{document}